\documentclass[journal]{IEEEtran}
\usepackage[cmex10]{amsmath}
\usepackage{amssymb,amsthm}

\newcommand{\BIBD}{\operatorname{BIBD}}
\newcommand{\Tr}{\operatorname{Tr}}
\newcommand{\tr}{\operatorname{tr}}
\newcommand{\Span}{\operatorname{span}}

\newcommand{\bbC}{\mathbb{C}}
\newcommand{\bbF}{\mathbb{F}}
\newcommand{\bbH}{\mathbb{H}}
\newcommand{\bbR}{\mathbb{R}}
\newcommand{\bbZ}{\mathbb{Z}}

\newcommand{\bfc}{\mathbf{c}}
\newcommand{\bfC}{\mathbf{C}}
\newcommand{\bfd}{\mathbf{d}}
\newcommand{\bfD}{\mathbf{D}}
\newcommand{\bfE}{\mathbf{E}}
\newcommand{\bfH}{\mathbf{H}}
\newcommand{\bfI}{\mathbf{I}}
\newcommand{\bfJ}{\mathbf{J}}
\newcommand{\bfP}{\mathbf{P}}

\newcommand{\bfs}{\mathbf{s}}
\newcommand{\bfS}{\mathbf{S}}
\newcommand{\bfU}{\mathbf{U}}
\newcommand{\bfV}{\mathbf{V}}
\newcommand{\bfx}{\mathbf{x}}
\newcommand{\bfX}{\mathbf{X}}
\newcommand{\bfy}{\mathbf{y}}
\newcommand{\bfY}{\mathbf{Y}}
\newcommand{\bfZ}{\mathbf{Z}}

\newcommand{\bfzero}{\mathbf{0}}
\newcommand{\bfone}{\mathbf{1}}
\newcommand{\bfdelta}{\boldsymbol{\delta}}
\newcommand{\bfPi}{\boldsymbol{\Pi}}
\newcommand{\bfSigma}{\boldsymbol{\Sigma}}
\newcommand{\bfphi}{\boldsymbol{\varphi}}
\newcommand{\bfPhi}{\boldsymbol{\Phi}}
\newcommand{\bfPsi}{\boldsymbol{\Psi}}

\newcommand{\calB}{\mathcal{B}}
\newcommand{\calD}{\mathcal{D}}
\newcommand{\calG}{\mathcal{G}}
\newcommand{\calV}{\mathcal{V}}

\newcommand{\rmE}{\mathrm{E}}
\newcommand{\rmi}{\mathrm{i}}
\newcommand{\rmT}{\mathrm{T}}

\newcommand{\abs}[1]{|{#1}|}
\newcommand{\set}[1]{\{{#1}\}}
\newcommand{\norm}[1]{\|{#1}\|}
\newcommand{\ip}[2]{\langle{#1},{#2}\rangle}
\newcommand{\romani}{\renewcommand{\labelenumi}{(\roman{enumi})}}

\newtheorem{theorem}{Theorem}
\newtheorem{corollary}{Corollary}
\newtheorem{lemma}{Lemma}

\theoremstyle{definition}
\newtheorem{definition}{Definition}
\newtheorem{problem}{Problem}

\begin{document}
\title{Equiangular Tight Frames from Hyperovals}

\author{Matthew~Fickus~\IEEEmembership{Member,~IEEE}, Dustin~G.~Mixon and John~Jasper%
\thanks{M.~Fickus and D.~G.~Mixon are with the Department of Mathematics and Statistics, Air Force Institute of Technology, Wright-Patterson Air Force Base, OH 45433, USA, e-mail: Matthew.Fickus@gmail.com.}%
\thanks{J.~Jasper is with the Department of Mathematical Sciences, University of Cincinnati, Cincinnati, OH 45221, USA.}%
\thanks{This paper was presented in part at the 2016 Spring Central Sectional Meeting of the American Mathematical Society.}%
}

\maketitle

\begin{abstract}
An equiangular tight frame (ETF) is a set of equal norm vectors in a Euclidean space whose coherence is as small as possible, equaling the Welch bound.
Also known as Welch-bound-equality sequences, such frames arise in various applications, such as waveform design, quantum information theory, compressed sensing and algebraic coding theory.
ETFs seem to be rare, and only a few methods of constructing them are known.
In this paper, we present a new infinite family of complex ETFs that arises from hyperovals in finite projective planes.
In particular, we give the first ever construction of a complex ETF of 76 vectors in a space of dimension 19.
Recently, a computer-assisted approach was used to show that real ETFs of this size do not exist,
resolving a longstanding open problem in this field.
Our construction is a modification of a previously known technique for constructing ETFs from balanced incomplete block designs.
\end{abstract}


\section{Introduction}

An equiangular tight frame is a type of optimal packing of lines in Euclidean space.
To be precise, let $d\leq n$ be positive integers, and let $\bbH_d$ be a real or complex Hilbert space of dimension $d$.
Welch~\cite{Welch74} gives a lower bound on the \textit{coherence} of any sequence $\set{\bfphi_i}_{i=1}^n$ of $n$ equal norm vectors in $\bbH_d$:
\begin{equation}
\label{equation.Welch bound}
\max_{i\neq j}\frac{\abs{\ip{\bfphi_i}{\bfphi_j}}}{\norm{\bfphi_i}\norm{\bfphi_j}}\geq\bigg[\frac{n-d}{d(n-1)}\biggr]^{\frac12}.
\end{equation}
It is well-known~\cite{StrohmerH03} that this lower bound is achieved if and only if $\set{\bfphi_i}_{i=1}^n$ is an \textit{equiangular tight frame} (ETF) for $\bbH_d$.

For example, when $d=6$ and $n=16$, the Welch bound~\eqref{equation.Welch bound} is $\frac13$,
and it is achieved by certain special choices of $16$ vectors in $\bbR^6$,
such as the columns of the following matrix:
\begin{equation}
\setlength{\arraycolsep}{2pt}
\label{equation.6x16 Steiner}
\left[\begin{array}{cccccccccccccccc}
+&-&+&-&+&-&+&-&0&0&0&0&0&0&0&0\\
0&0&0&0&0&0&0&0&+&-&+&-&+&-&+&-\\
+&+&-&-&0&0&0&0&+&+&-&-&0&0&0&0\\
0&0&0&0&+&+&-&-&0&0&0&0&+&+&-&-\\
+&-&-&+&0&0&0&0&0&0&0&0&+&-&-&+\\
0&0&0&0&+&-&-&+&+&-&-&+&0&0&0&0
\end{array}\right].
\end{equation}
Here, ``$+$" and ``$-$" denote $1$ and $-1$, respectively.

Having minimal coherence, ETFs are useful in a number of real-world applications, including waveform design for wireless communication~\cite{StrohmerH03}, compressed sensing~\cite{BajwaCM12,BandeiraFMW13}, quantum information theory~\cite{RenesBSC04,Zauner99} and algebraic coding theory~\cite{JasperMF14}.
Unfortunately, ETFs also seem to be rare.
In particular, when $\bbH_d$ is a real Hilbert space, $n$ and $d$ necessarily satisfy certain strong integrality conditions~\cite{SustikTDH07}.
Moreover, only a few methods for constructing infinite families of ETFs are known,
and with the exception of orthonormal bases and regular simplices,
each of these methods involve some type of combinatorial design.

Real ETFs in particular are equivalent to special types of \textit{strongly regular graphs} (SRGs)~\cite{HolmesP04,Waldron09}.
Such graphs have a rich literature~\cite{Brouwer07,Brouwer15,CorneilM91}.
The interrelated concepts of conference matrices, Hadamard matrices, Paley tournaments, and quadratic residues have all been used to construct various infinite families of ETFs in which $n$ is either nearly or exactly $2d$~\cite{StrohmerH03,HolmesP04,Renes07,Strohmer08}.
Other constructions offer much more flexibility regarding the size of $\frac nd$.
These include \textit{harmonic ETFs}, which are obtained by restricting the Fourier basis on a finite abelian group to a \textit{difference set} for that group~\cite{StrohmerH03,XiaZG05,DingF07}, and \textit{Steiner ETFs}, which arise from a tensor-like product of a simplex and the incidence matrix of a certain type of \textit{balanced incomplete block design}~\cite{FickusMT12,JasperMF14}.
To be clear, many of these ideas are rediscoveries or reimaginings of more classical results.
In particular, see~\cite{Rankin56}, \cite{Turyn65} and \cite{GoethalsS70} for precursors of the Welch bound, harmonic ETFs and Steiner ETFs, respectively.

In this paper, we generalize the Steiner ETF construction of~\cite{FickusMT12,JasperMF14} to construct a new infinite family of complex ETFs.
In particular for any $e\geq1$, we construct an $n$-vector ETF for a complex $d$-dimensional Hilbert space where
\begin{equation}
\label{equation.new ETF parameters}
d=2^{2e}+2^e-1,
\quad
n=2^e(2^{2e}+2^e-1).
\end{equation}

From our perspective, this construction is significant for two reasons.
First, as noted above, ETFs seem to be rare, and few infinite families of them are known.
Comparing against known constructions~\cite{FickusM15}, it appears that with the exception of $e=1$ case,
no real or complex ETF with parameters~\eqref{equation.new ETF parameters} has been discovered before.
In fact, taking $e=2$ in \eqref{equation.new ETF parameters} gives $d=19$ and $n=76$,
and the existence of a real ETF of this size appears to have been recently ruled out with a computer-assisted search~\cite{AzarijaM15},
resolving a longstanding open problem~\cite{Yu15}.

Second, though the construction itself is a generalization of Steiner ETFs~\cite{FickusMT12}, it exploits a new realization:
like previous work, we construct $n$-vector ETFs for $\bbH_d$ as the columns of an $m\times n$ matrix;
unlike previous work, we realize that sometimes, the most natural choice of $m$ is neither $d$ nor $n$.
For example,
the columns of the following $6\times 10$ matrix form a $10$-vector ETF for a $5$-dimensional Hilbert space, namely the orthogonal complement of the all-ones vector $\bfone$ in $\bbR^6$:
\begin{equation}
\setlength{\arraycolsep}{2pt}
\label{equation.5x10 DOE}
\left[\begin{array}{rrrrrrrrrr}
+&+&+&+&+&+&+&+&+&+\\
+&+&+&+&-&-&-&-&-&-\\
+&-&-&-&+&+&+&-&-&-\\
-&+&-&-&+&-&-&+&+&-\\
-&-&+&-&-&+&-&+&-&+\\
-&-&-&+&-&-&+&-&+&+
\end{array}\right].
\end{equation}
This can be verified by noting that all the columns of this matrix are orthogonal to $\bfone$
and that together they achieve equality in~\eqref{equation.Welch bound} for $n=10$ and $d=5$.
Inspired by problems in experimental design,
this example was produced by taking all $\binom{6}{3}=20$ $\set{\pm1}$-valued vectors in $\bbR^6$ that are orthogonal to $\bfone$,
and discarding one vector from each pair of antipodes.
Our new construction is a generalization of this example: we form ETFs of size~\eqref{equation.new ETF parameters} as the columns of an $m\times n$ matrix where $m=d+1=2^{2e}+2^e$.
The construction only happens to be real in the $e=1$ case, namely~\eqref{equation.5x10 DOE}.

In the next section, we introduce the background material we need for the rest of the paper.
In Section~III we show how to construct complex ETFs of size~\eqref{equation.new ETF parameters}.
The construction relies on special collections of lines in finite affine planes that arise from \textit{hyperovals} in corresponding projective planes.
In the fourth and final section, we discuss how these ETFs can be made \textit{flat}.
This leads to several interesting open problems.

\section{Background}

\subsection{Steiner equiangular tight frames}
Throughout, let the field of scalars $\bbF$ be either $\bbR$ or $\bbC$, and let $\bbH_d$ be a $d$-dimensional inner product space over $\bbF$.
The \textit{synthesis operator} of a finite sequence of vectors $\set{\bfphi_i}_{i=1}^{n}$ in $\bbH_d$ is $\bfPhi:\bbF^n\rightarrow\bbH_d$, $\bfPhi\bfy:=\sum_{i=1}^{n}\bfy(i)\bfphi_i$.
The adjoint of $\bfPhi$ is the \textit{analysis operator} $\bfPhi^*:\bbH_d\rightarrow\bbF^n$, $(\bfPhi^*\bfx)(n)=\ip{\bfphi_n}{\bfx}$.
Composing these two operators gives the \textit{frame operator} $\bfPhi\bfPhi^*:\bbH_d\rightarrow\bbH_d$,
$\bfPhi\bfPhi^*\bfx=\sum_{i=1}^{n}\ip{\bfphi_i}{\bfx}\bfphi_i$ and the $n\times n$ \textit{Gram matrix} $\bfPhi^*\bfPhi$ whose $(i,j)$th entry is $\ip{\bfphi_i}{\bfphi_j}$.

A sequence of vectors $\set{\bfphi_i}_{i=1}^{n}$ is a \textit{tight frame} for $\bbH_d$ if there exists $a>0$ such that $\bfPhi\bfPhi^*=a\bfI$.
When $a$ is specified, it is an \textit{$a$-tight frame}.
It is \textit{equal norm} if there exists $r>0$ such that $\norm{\bfphi_i}^2=r$ for all $i$,
and is \textit{equiangular} if it is equal norm and there exists $w\geq0$ such that $\abs{\ip{\bfphi_i}{\bfphi_j}}=w$ for all $i\neq j$.
When $\set{\bfphi_i}_{i=1}^{n}$ is both equiangular and a tight frame, it is an \textit{equiangular tight frame} (ETF).
It is well known~\cite{StrohmerH03} that equal norm vectors $\set{\bfphi_i}_{i=1}^{n}$ in $\bbH_d$ achieve equality in~\eqref{equation.Welch bound} if and only if they form an ETF for $\bbH_d$;
see Lemma~\ref{lemma.tight for subspace} below for a generalized version of this result.

In the special case where $\bbH_d$ is $\bbF^d$, the synthesis operator of $\set{\bfphi_i}_{i=1}^{n}$ is simply the $d\times n$ matrix $\bfPhi$ whose $i$th column is $\bfphi_i$.
In this case, $\set{\bfphi_i}_{i=1}^{n}$ is an ETF for $\bbF^d$ if and only if the rows of $\bfPhi$ are orthogonal and have constant norm (tightness) and the columns of $\bfPhi$ are equiangular.
As we now discuss, Steiner ETFs~\cite{FickusMT12} are one way to directly construct such matrices;
in Section~III, we generalize this construction.

A \textit{balanced incomplete block design} (BIBD) is a finite set $\calV$---whose elements are called \textit{vertices}---along with any set $\calB$ of subsets of $\calV$---whose elements are called \textit{blocks}---for which there exists positive integers $\lambda$, $k$, $r$ such that:
\begin{enumerate}
\romani
\item
every block contains exactly $k$ vertices,
\item
every vertex is contained in exactly $r$ blocks,
\item
any pair of vertices is contained in exactly $\lambda$ blocks.
\end{enumerate}
Letting $v$ and $b$ denote the desired number of vertices and blocks, respectively, a $\set{0,1}$-valued $b\times v$ matrix $\bfX$ is an incidence matrix of a BIBD precisely when
\begin{equation}
\label{equation.BIBD properties}
\bfX\bfone=k\bfone,
\quad
\bfX^\rmT\bfX=(r-\lambda)\bfI+\lambda\bfJ,
\end{equation}
for some integers $k$, $r$ and $\lambda$.
Here and throughout, $\bfone$ and $\bfJ$ denote all-ones vectors and matrices, respectively.
Note the second condition in~\eqref{equation.BIBD properties} implies $\bfone^\rmT\bfX=r\bfone^\rmT$.
Moreover, $v$, $k$ and $\lambda$ determine $r$ and $b$ according to
\begin{equation}
\label{equation.BIBD parameters}
bk=vr,
\quad
\lambda(v-1)=r(k-1).
\end{equation}
To see this, note multiplying $\bfX$ on the left and right by $\bfone^\rmT$ and $\bfone$, respectively, gives the first identity,
while doing the same to the equation $\bfX^\rmT\bfX=(r-\lambda)\bfI+\lambda\bfJ$ gives the second.
Because of this, such incidence structures are often called a \textit{$2$-$(v,k,\lambda)$ design}, and are denoted as a $\BIBD(v,k,\lambda)$.

Steiner ETFs arise from BIBDs with $\lambda=1$, which are also known as $(2,k,v)$-\textit{Steiner systems}.
Having $\lambda=1$ means that any two distinct vertices determine a unique block,
and thus such BIBDs are a type of finite geometry.
Indeed the canonical examples of such BIBDs are finite \textit{affine} and \textit{projective planes} of \textit{order} $q\geq 2$,
namely $\BIBD(q^2,q,1)$ and $\BIBD(q^2+q+1,q+1,1)$, respectively.
As detailed later on in this section, such designs are known to exist whenever $q$ is the power of a prime.
When $q=2$, we have the following incidence matrices $\bfX$, for example:
\begin{equation}
\setlength{\arraycolsep}{2pt}
\label{equation.6x4 and 7x7 BIBDs}
\left[\begin{array}{rrrr}
1&1&0&0\\
0&0&1&1\\
1&0&1&0\\
0&1&0&1\\
1&0&0&1\\
0&1&1&0
\end{array}\right],\quad
\left[\begin{array}{rrrrrrr}
1&1&0&0&1&0&0\\
0&0&1&1&1&0&0\\
1&0&1&0&0&1&0\\
0&1&0&1&0&1&0\\
1&0&0&1&0&0&1\\
0&1&1&0&0&0&1\\
0&0&0&0&1&1&1
\end{array}\right].
\end{equation}

A Steiner ETF is formed by taking a tensor-like product of the $b\times v$ incidence matrix of a $\BIBD(v,k,1)$ and a (possibly complex) Hadamard matrix of size $r+1$.
To define them rigorously,
we borrow the following concept from~\cite{FickusJMP16}:

\begin{definition}
\label{definition.embedding}
Let $\bfX$ be a $b\times v$ incidence matrix of a $\BIBD(v,k,1)$.
For any $j=1,\dotsc,v$, a corresponding \textit{embedding} is an operator $\bfE_j:\bbF^r\rightarrow\bbF^b$ that maps the standard basis of $\bbF^r$ to the standard basis of the $r$-dimensional subspace of $\bbF^b$ that consists of vectors supported on $\set{i: \bfX(i,j)=1}$.
\end{definition}

For example, for the $6\times 4$ incidence matrix in~\eqref{equation.6x4 and 7x7 BIBDs} we can take $\bfE_1$, $\bfE_2$, $\bfE_3$ and $\bfE_4$ to be
\begin{equation}
\setlength{\arraycolsep}{2pt}
\label{equation.6x3 embedding operators}
\left[\begin{array}{rrr}
1&0&0\\
0&0&0\\
0&1&0\\
0&0&0\\
0&0&1\\
0&0&0
\end{array}\right],\
\left[\begin{array}{rrr}
1&0&0\\
0&0&0\\
0&0&0\\
0&1&0\\
0&0&0\\
0&0&1
\end{array}\right],\
\left[\begin{array}{rrr}
0&0&0\\
1&0&0\\
0&1&0\\
0&0&0\\
0&0&0\\
0&0&1
\end{array}\right],\
\left[\begin{array}{rrr}
0&0&0\\
1&0&0\\
0&0&0\\
0&1&0\\
0&0&1\\
0&0&0
\end{array}\right],
\end{equation}
respectively.
A Steiner ETF is formed by using these operators to embed $v$ copies of a unimodular simplex for $\bbF^r$ into $\bbF^b$.
To be precise, let $\set{\bfs_l}_{l=1}^{r+1}$ be the columns of an $r\times(r+1)$ submatrix $\bfS$ of an $(r+1)\times(r+1)$ Hadamard matrix whose entries lie in $\bbF$.
For example, the $6\times 4$ incidence matrix in~\eqref{equation.6x4 and 7x7 BIBDs} has $r=3$ and we can form $\bfS$ by removing the first row from the canonical $4\times 4$ Hadamard matrix:
\begin{equation}
\setlength{\arraycolsep}{2pt}
\label{equation.3x4 unimodular simplex}
\bfS
=\left[\begin{array}{rrrr}
+&-&+&-\\
+&+&-&-\\
+&-&-&+
\end{array}\right].
\end{equation}
For any $\BIBD(v,k,1)$ and unimodular simplex $\set{\bfs_l}_{l=1}^{r+1}$,
the corresponding Steiner ETF is $\set{\bfE_j^{}\bfs_l}_{j=1,}^{v}\,_{l=1}^{r+1}$.
As shown in~\cite{FickusMT12,FickusJMP16}, these $v(r+1)$ vectors form an ETF for $\bbF^b$,
and their synthesis operator is the block matrix
\begin{equation*}
\bfPhi
=\left[\begin{array}{ccc}\bfE_1\bfS&\cdots&\bfE_v\bfS\end{array}\right].
\end{equation*}
For example, using~\eqref{equation.6x3 embedding operators} to embed $4$ copies of the unimodular simplex in~\eqref{equation.3x4 unimodular simplex} gives the $16$-vector ETF for $\bbR^6$ given in~\eqref{equation.6x16 Steiner}.
The main idea of the construction is that $\ip{\bfE_j\bfs_l}{\bfE_{j'}\bfs_{l'}}$ has unit modulus whenever $(j,l)\neq(j',l')$:
if $j\neq j'$, this follows from the fact that any two distinct blocks in the BIBD have exactly one vertex in common;
if $j=j'$ but $l\neq l'$, this follows from the fact that the inner product of any two distinct columns of $\bfS$ has unit modulus.
A Steiner ETF can be real whenever there exists a real Hadamard matrix of size $r+1$.
The famously-unresolved Hadamard conjecture proposes that this happens precisely when $r+1$ is divisible by $4$.

\subsection{Finite affine and projective planes}

In the next section, we generalize the Steiner ETF construction in a way that yields complex ETFs with parameters~\eqref{equation.new ETF parameters}.
This construction relies on some unusually nice properties of finite affine and projective planes of order $q$,
namely of $\BIBD(q^2,q,1)$ and $\BIBD(q^2+q+1,q+1,1)$, respectively.
(For the sake of brevity, we omit the standard, axiomatic definitions of these geometries).

We will need the fact that any affine plane of order $q$ is a subincidence structure of a projective plane of order $q$.
This follows from a type of parallel postulate that any $\BIBD(v,k,1)$ must satisfy.
To be precise, let $\bfX$ be the incidence matrix of such a BIBD, and note the $(i,i')$th entry of $(k-1)\bfI+\bfJ-\bfX\bfX^\rmT$ is $1$ when the $i$th and $i'$th blocks are disjoint, and is otherwise $0$.
The number of blocks disjoint from block $i$ that contain vertex $j$ is thus given by the $(i,j)$th entry of
$[(k-1)\bfI+\bfJ-\bfX\bfX^\rmT]\bfX=(r-k)(\bfJ-\bfX)$;
here we have used~\eqref{equation.BIBD parameters} and~\eqref{equation.BIBD properties} to simplify.
As such, if a block does not contain a vertex, there are exactly $r-k$ blocks disjoint from it that do contain that vertex.

In particular, for a $\BIBD(q^2,q,1)$ we have $r-k=1$ and so any vertex not in a block is contained in exactly one disjoint block.
This implies that if blocks $i$ and $i'$ are disjoint, and blocks $i'$ and $i''$ are disjoint,
then blocks $i$ and $i''$ are either equal or disjoint, or else any point in their intersection is not contained in exactly one block disjoint from block $i'$.
Saying two blocks are \textit{parallel} when they are either equal or disjoint,
we thus have that parallelism is an equivalence relation on the blocks of an affine plane.
As a consequence, any affine plane is \textit{resolvable},
meaning its $b=q(q+1)$ blocks can be arranged as $r=q+1$ groupings of $\frac vk=q$ blocks, each forming a partition for the set of vertices.
For example, the incidence matrix of the affine plane of order $2$ given in~\eqref{equation.6x4 and 7x7 BIBDs} is arranged so that its first and second blocks partition its vertices, as does its third and fourth blocks, and its fifth and sixth blocks.

The fact that any affine plane is resolvable allows it to be extended to a projective plane of the same order.
To do this, we append each block in any of the $q+1$ parallel classes with a new ``vertex at infinity" that is unique to that class,
and create a single new ``block at infinity" that contains the $q+1$ new vertices.
For example, extending the $6\times 4$ incidence matrix of~\eqref{equation.6x4 and 7x7 BIBDs} in this manner yields the $7\times 7$ matrix to its right.
Conversely, any projective plane contains many affine planes:
choose any block, and remove it as well as all vertices it contains.

Projective planes are \textit{symmetric} BIBDs, meaning $v=b$, or equivalently that $k=r$.
In general, the \textit{dual} of a $\BIBD(v,k,\lambda)$ is the incidence structure obtained by interchanging the roles of vertices and blocks, or equivalently, taking the transpose of its incidence matrix $\bfX$.
Here, \eqref{equation.BIBD parameters} implies \smash{$r^2-b\lambda=\frac{r^2(v-k)}{k(v-1)}\geq0$},
at which point~\eqref{equation.BIBD properties} implies the columns of $(r-\lambda)^{-\frac12}\{\bfX-\tfrac1b[r\pm(r^2-b\lambda)^\frac12]\bfJ\}$ are orthonormal.
When the BIBD is symmetric, this matrix is square and so its rows are also necessarily orthonormal, implying $\bfX\bfX^\rmT=(k-\lambda)\bfI+\lambda\bfJ$.
Thus, the dual of a symmetric $\BIBD(v,k,\lambda)$ is another $\BIBD(v,k,\lambda)$.
In particular, the dual of any projective plane is another projective plane.

Whenever $q$ is the power of a prime, there is a canonical construction of affine and projective planes of order $q$.
In particular, letting $\bbF_q$ be the field of $q$ elements,
we can form an affine plane by letting $\calV=\bbF_q^2$ and letting $\calB$ be the collection of all affine lines in this vector space,
namely sets of the form
\begin{equation}
\label{equation.affine block}
\set{(x,y)\in\bbF_q^2: d x+e y+f=0}
\end{equation}
for some $d,e,f\in\bbF_q$ where $(d,e)\neq(0,0)$.
Though this manner of parametrizing these lines is redundant---multiplying $(d,e,f)$ by a nonzero scalar yields the same line---it gives a simple way to produce parallel classes: fix $(d,e)\neq(0,0)$ and vary $f$.
It also facilitates the embedding of this affine plane in a projective plane where vertices and blocks correspond to one- and two-dimensional subspaces of $\bbF_q^3$, respectively.

To be precise, let $[x,y,z]$ denote the set of all nonzero scalar multiples of a nonzero vector $(x,y,z)\in\bbF_q^3$.
To form a projective plane, let $\calV$ be the set of all such $[x,y,z]$,
and let $\calB$ be all sets of the form
\begin{equation}
\label{equation.projective block}
\set{[x,y,z]: d x+e y+f z=0}
\end{equation}
for some $[d,e,f]$.
The mapping $(x,y)\mapsto[x,y,1]$ embeds our canonical affine geometry into this projective geometry:
for any $[d,e,f]$ with $(d,e)\neq(0,0)$, \eqref{equation.projective block} is the image of \eqref{equation.affine block} under this mapping, along with the ``vertex at infinity" $[-e,d,0]$.
Note that as expected, each new vertex corresponds to a unique parallel class in the affine plane,
and taken together they constitute a new ``block at infinity," namely~\eqref{equation.projective block} where $[d,e,f]=[0,0,1]$.

Here, we can identify $\bbF_q^3$ with the additive group of the field $\bbF_{q^3}$.
To be precise, let $\alpha$ be a generator of the cyclic multiplicative group $\bbF_{q^3}^\times$, and let $(x,y,z):=x+y\alpha+z\alpha^2$ for all $x,y,z\in\bbF_q$.
Since each vertex $[x,y,z]$ consists of all nonzero scalar multiples of the nonzero vector $(x,y,z)$,
it is a unique element of the quotient group $\bbF_{q^3}^\times/\bbF_q^\times$.
That is,
\begin{equation*}
\calV
=\bbF_{q^3}^\times/\bbF_q^\times
=\langle\alpha\rangle/\langle\alpha^{q^2+q+1}\rangle
\cong\bbZ_{q^2+q+1}.
\end{equation*}
Moreover, each two-dimensional subspace of $\bbF_{q^3}$ is the kernel of a mapping of the form $\beta\mapsto\tr(\alpha^{-l}\beta)$ where $\tr:\bbF_{q^3}\rightarrow\bbF_q$,
\smash{$\tr(\beta):=\beta+\beta^q+\beta^{q^2}$}, is the \textit{field trace}.
As such, all blocks correspond to translates of the subset of $\bbZ_{q^2+q+1}$ which correspond to members of $\bbF_{q^3}^\times/\bbF_q^\times$ whose coset representatives have trace $0$.
The resulting incidence matrix is circulant, and corresponds to a \textit{Singer difference set}.

These explicit constructions show that affine and projective planes of order $q$ exist whenever $q$ is a power of a prime.
Conversely, it is famously conjectured that if there exists an affine or projective plane of order $q$ then $q$ is necessarily a prime power.
This conjecture remains open for many values of $q$, such as $12$.
We also note that for some prime powers $q$, there are constructions of projective planes of order $q$ that are provably not \textit{Desarguesian}, meaning they are not isomorphic to the canonical construction given above.
Our new method for constructing complex ETFs can be applied to any projective plane of order $q$ that contains a \textit{hyperoval}, defined below.
Applying this construction to the canonical projective plane of order $q=2^e$ yields complex ETFs with parameters~\eqref{equation.new ETF parameters}.
If non-prime-power-order examples of such projective planes are discovered in the future, this same construction will yield yet more complex ETFs.

\section{New ETFs from hyperovals}

In this section, we show how to construct ETFs with parameters~\eqref{equation.new ETF parameters}.
As detailed below, the main idea is that any projective plane whose dual contains a hyperoval also contains an affine plane that contains the dual of a $\BIBD(q+1,2,1)$.
This permits a new Steiner-like ETF construction involving Hadamard matrices of two distinct sizes.
Some of the novelty here is that these ETFs most naturally arise in non-obvious subspaces of $\bbF^m$.
The following result characterizes such frames in general.

\begin{lemma}
\label{lemma.tight for subspace}
For any vectors $\set{\bfphi_i}_{i=1}^{n}$ in $\bbF^m$, let $\bfPhi$ be the $m\times n$ matrix whose $i$th column is $\bfphi_i$ for all $i$.
For any $a>0$,
the following are equivalent:
\begin{enumerate}
\romani
\item
$\set{\bfphi_i}_{i=1}^{n}$ forms an $a$-tight frame for its span,
\item
$\bfPhi\bfPhi^*\bfPhi=a\bfPhi$,
\item
$(\bfPhi\bfPhi^*)^2=a\bfPhi\bfPhi^*$,
\item
$(\bfPhi^*\bfPhi)^2=a\bfPhi^*\bfPhi$.
\end{enumerate}

Also, if $\set{\bfphi_i}_{i=1}^{n}$ is contained in a $d$-dimensional subspace $\bbH_d$ of $\bbF^m$, then it forms an $a$-tight frame for $\bbH_d$ if and only if  $\bfPhi\bfPhi^*=a\bfPi$ where $\bfPi$ is the $m\times m$ orthogonal projection matrix onto $\bbH_d$.

As such, $\set{\bfphi_i}_{i=1}^{n}$ forms an ETF for its span if and only if (ii)--(iv) hold and $\set{\bfphi_i}_{i=1}^{n}$ is equiangular.
In this case, letting $r=\norm{\bfphi_i}^2$, the dimension $d$ can be computed from either the tight frame constant or the equiangularity constant:
\begin{equation*}
a=\frac{rn}d,
\quad
w=r\biggl[\frac{n-d}{d(n-1)}\biggr]^{\frac12}.
\end{equation*}

Alternatively, equal norm vectors $\set{\bfphi_i}_{i=1}^{n}$ in a subspace $\bbH_d$ of $\bbF^m$ of dimension $d$ form an ETF for $\bbH_d$ if and only if they achieve equality in~\eqref{equation.Welch bound}.
\end{lemma}

\begin{proof}
Fix $\set{\bfphi_i}_{i=1}^{n}$ in $\bbF^m$ and $a>0$,
and let $\bbH_d$ be any $d$-dimensional subspace of $\bbF^m$ that contains $\set{\bfphi_i}_{i=1}^{n}$.

To show (i) is equivalent to (ii),
note the codomain of the synthesis operator $\bfy\mapsto\bfPhi\bfy$ can either be regarded as $\bbF^m$ or $\bbH_d$,
and that by definition, $\set{\bfphi_i}_{i=1}^{n}$ forms an $a$-tight frame for $\bbH_d$ if and only if the restricted version of this operator, namely $\bfPhi:\bbF^n\rightarrow\bbH_d$, satisfies $\bfPhi\bfPhi^*=a\bfI$.
When regarding $\bfPhi$ as an $m\times n$ matrix,
this is equivalent to having $\bfPhi\bfPhi^*\bfx=a\bfx$ for all $\bfx\in\bbH_d$.
In the special case where $\bbH_d$ is taken to be $\Span\set{\bfphi_i}_{i=1}^{n}=\set{\bfPhi\bfy: \bfy\in\bbF^n}$,
this tells us (i) is equivalent to having $\bfPhi\bfPhi^*\bfPhi\bfy=a\bfPhi\bfy$ for all $\bfy\in\bbF^n$, namely (ii).

More generally, writing $\bbH_d=\set{\Pi\bfx: \bfx\in\bbF^m}$, we see that $\set{\bfphi_i}_{i=1}^{n}$ forms an $a$-tight frame for $\bbH_d$ if and only if $\bfPhi\bfPhi^*\bfPi\bfx=a\bfPi\bfx$ for all $\bfx\in\bbF^m$.
To simplify this further, note that since $\bfPi\bfphi_i=\bfphi_i$ for all $i$, we have $\bfPi\bfPhi=\bfPhi$ and so $\bfPhi^*\bfPi=(\bfPi\bfPhi)^*=\bfPhi^*$.
Thus, $\set{\bfphi_i}_{i=1}^{n}$ forms an $a$-tight frame for $\bbH_d$ if and only if $\bfPhi\bfPhi^*=a\bfPi$, as claimed.
We emphasize that the above argument is delicate, and fails if we do not assume that each $\bfphi_i$ lies in $\bbH_d$.
Indeed, without this assumption, it is possible for $\set{\bfphi_i}_{i=1}^{n}$ to not form an $a$-tight frame for $\bbH_d$ and yet still have $\bfPhi\bfPhi^*\bfx=a\bfx$ for all $\bfx\in\bbH_d$, namely tight \textit{pseudoframes}~\cite{LiO04}.

Continuing, note that multiplying (ii) by $\bfPhi^*$ on the right or left gives (iii) and (iv), respectively.
Conversely, taking the singular value decomposition $\bfPhi=\bfU\bfSigma\bfV^*$, if either (iii) or (iv) hold then the diagonal entries of the diagonal block of $\Sigma$ are all either $0$ or $\sqrt{a}$, implying $\Sigma\Sigma^*\Sigma=\Sigma$ and so (ii).
To summarize, we have proven that (i)--(iv) are equivalent as well as the second claim.

For the remaining conclusions, we further assume that $\norm{\bfphi_i}^2=r$ for all $i$,
and generalize an argument of~\cite{JasperMF14}.
If $\set{\bfphi_i}_{i=1}^{n}$ is an $a$-tight frame for $\bbH_d$, cycling a trace gives
\begin{equation*}
a d
=\Tr(a\bfPi)
=\Tr(\bfPhi\bfPhi^*)
=\Tr(\bfPhi^*\bfPhi)
=rn,
\end{equation*}
and so \smash{$a=\frac{rn}d$}.
As such, one way to measure the tightness of $\set{\bfphi_i}_{i=1}^{n}$ is to take the Frobenius norm of $\bfPhi\bfPhi^*-\tfrac{rn}d\bfPi$.
We now simplify and bound this quantity using trace properties:
\begin{align}
0
\nonumber&\leq\Tr[(\bfPhi\bfPhi^*-\tfrac{rn}d\bfPi)^2]\\
\nonumber&=\Tr(\bfPhi^*\bfPhi\bfPhi^*\bfPhi)-2\tfrac{rn}d\Tr(\bfPhi^*\bfPhi)+(\tfrac{rn}{d})^2\Tr(\bfPi)\\
\nonumber&=\sum_{i=1}^{n}\sum_{j=1}^{n}\abs{\ip{\bfphi_i}{\bfphi_j}}^2-\tfrac{r^2n^2}d\\
\nonumber&=r^2\sum_{i=1}^{n}\sum_{\substack{j=1\\j\neq i}}^{n}\Bigl(\tfrac{\abs{\ip{\bfphi_i}{\bfphi_j}}}{\norm{\bfphi_i}\norm{\bfphi_j}}\Bigr)^2-r^2n(\tfrac nd-1)\\
\label{equation.proof of Welch bound}
&\leq r^2n(n-1)\Bigl(\max_{i\neq j}\tfrac{\abs{\ip{\bfphi_i}{\bfphi_j}}}{\norm{\bfphi_i}\norm{\bfphi_j}}\Bigr)^2-r^2\tfrac{n(n-d)}{d}.
\end{align}
Solving for the coherence here gives the Welch bound~\eqref{equation.Welch bound}.
Moreover, if $\set{\bfphi_i}_{i=1}^{n}$ is an ETF for $\bbH_d$ with $\abs{\ip{\bfphi_i}{\bfphi_j}}=w$ for all $i\neq j$, the two above inequalities become equalities, giving equality in~\eqref{equation.Welch bound} and \smash{$w=r[\tfrac{n-d}{d(n-1)}]^{\frac12}$}.
Conversely, if~\eqref{equation.Welch bound} holds with equality, the final quantity in~\eqref{equation.proof of Welch bound} is $0$,
implying both inequalities are equalities and so $\set{\bfphi_i}_{i=1}^{n}$ is tight and equiangular.
\end{proof}

These characterizations in hand, we turn to the construction itself.
A \textit{hyperoval} in a projective plane of order $q$ is a set of $q+2$ vertices where no three of these vertices lie in a common block.
For example, the first four vertices (columns) of the projective plane of order $2$ given in~\eqref{equation.6x4 and 7x7 BIBDs} are a hyperoval since no block (row) contains three of them.
As we now detail, hyperovals decompose a projective plane into several other incidence structures.
\begin{lemma}
\label{lemma.hyperoval decomposition}
A projective plane of order $q$ contains a hyperoval if and only if it has an incidence matrix of form
\begin{equation}
\label{equation.hyperoval decomposition}
\bfX=\left[\begin{array}{cc}\bfX_{1,1}&\bfX_{1,2}\\\bfzero&\bfX_{2,2}\end{array}\right]
\end{equation}
where $\bfX_{1,1}$ is the incidence matrix of a $\BIBD(q+2,2,1)$.
Here, the columns of $\bfX_{1,1}$ correspond to the vertices of the hyperoval.
Moreover, in this case, $q$ is even and $\bfX_{2,2}^\rmT$ is the incidence matrix of a $\BIBD(\frac12q(q-1),\frac12q,1)$.
\end{lemma}

\begin{proof}
Consider a $\BIBD(q^2+q+1,q+1,1)$ that contains $u$ vertices that are special in the sense that no three of them lie in a common block.
Without loss of generality, the first $u$ columns of the incidence matrix $\bfX$ correspond to these special vertices.
The total number of special vertex-block pairs is $\sum_{j=1}^{u}\sum_{i=1}^{b}\bfX(i,j)=\sum_{j=1}^{u}(q+1)=u(q+1)$.
At the same time, each of the $\binom{u}{2}$ pairs of distinct special vertices determine a unique block, and so
\begin{equation*}
u(q+1)=\sum_{i=1}^{b}\sum_{j=1}^{u}\bfX(i,j)\geq \binom{u}{2}2=u(u-1),
\end{equation*}
implying $u\leq q+2$.
Having $u=q+2$ gives equality above,
meaning all special vertex-block pairs occur in blocks containing exactly two special vertices.
That is, when the projective plane contains a hyperoval, there is an enumeration of its vertices and blocks so that its incidence matrix is of form~\eqref{equation.hyperoval decomposition} where $\bfX_{1,1}$ is the incidence matrix of a $\BIBD(q+2,2,1)$.
Conversely, when the projective plane has such an incidence matrix, its first $q+2$ vertices form a hyperoval since each block contains either $2$ or $0$ of them.

For the remaining conclusions, assume $\bfX_{1,1}$ is the incidence matrix of a $\BIBD(q+2,2,1)$.
Here, \eqref{equation.BIBD parameters} implies $\bfX_{1,1}$ is of size $\frac12(q+1)(q+2)\times(q+2)$,
implying the size of $\bfX_{2,2}$ is $\frac12q(q-1)\times(q^2-1)$.
Also, recall the dual of a projective plane is another projective plane, implying $\bfX\bfX^\rmT=q\bfI+\bfJ$ and so $\bfX_{2,2}^{}\bfX_{2,2}^\rmT=q\bfI+\bfJ$.
Thus, to show $\bfX_{2,2}^\rmT$ is the incidence matrix of $\BIBD(\frac12q(q-1),\frac12q,1)$ it suffices to show $\bfX_{2,2}^\rmT\bfone=\frac12q\bfone$.
To see this, note $\bfX_{1,1}\bfone=2\bfone$.
Also, the upper-right block of the equation $\bfX^\rmT\bfX=q\bfI+\bfJ$ gives $\bfX_{1,1}^\rmT\bfX_{1,2}^{}=\bfJ$,
where $\bfJ$ is $(q+2)\times(q^2-1)$.
Together, these facts imply
\begin{equation*}
(q+2)\bfone^\rmT
=\bfone^\rmT\bfJ=\bfone^\rmT\bfX_{1,1}^\rmT\bfX_{1,2}^{}
=2(\bfone^\rmT\bfX_{1,2}^{}).
\end{equation*}
Thus, $\bfone^\rmT\bfX_{1,2}^{}=\tfrac12(q+2)\bfone^\rmT$, which implies that $q$ is even.
Also, since $\bfone^\rmT\bfX=(q+1)\bfone^\rmT$ then
\begin{equation*}
(q+1)\bfone^\rmT
=\bfone^\rmT\bfX_{1,2}+\bfone^\rmT\bfX_{2,2}
=\tfrac12(q+2)\bfone^\rmT+\bfone^\rmT\bfX_{2,2}.
\end{equation*}
Thus, we indeed have $\bfX_{2,2}^\rmT\bfone=\frac12q\bfone$.
\end{proof}

The partition given in Lemma~\ref{lemma.hyperoval decomposition} has a geometric meaning:
the column partition indicates whether a vertex is in the hyperoval or not,
while the row partition indicates whether a block intersects the hyperoval in two or zero vertices,
namely whether the block is secant or exterior to the hyperoval.

The matrix $\bfX_{2,2}^\rmT$ given in Lemma~\ref{lemma.hyperoval decomposition} is the incidence matrix of a type of \textit{Denniston design}~\cite{Denniston69}.
As an example of this lemma, note the incidence matrix of the projective plane of order $2$ given in~\eqref{equation.6x4 and 7x7 BIBDs} is already in this form.
There, $\bfX_{2,2}$ is a $1\times 3$ matrix of ones, corresponding to a degenerate BIBD with $v=k=\lambda=1$ and $b=r=3$.

To prove our main result, we need an infinite family of matrices of form~\eqref{equation.hyperoval decomposition},
and thus need a general construction of hyperovals.
Recall that projective planes of order $q$ are only known to exist when $q$ is the power of a prime.
When coupled with the requirement that $q$ be even,
this means that all known constructions of hyperovals lie in projective planes of order $q=2^e$ for some $e\geq1$.
For the canonical projective plane of order $q=2^e$, the canonical hyperoval is set:
\begin{equation}
\label{equation.hyperoval construction}
\set{[t,t^2,1]: t\in\bbF_q}\cup\set{[0,1,0]}\cup\set{[1,0,0]}.
\end{equation}
No three of these vertices lie in a common block, since no three of the corresponding vectors are linearly dependent, a fact that follows from the corresponding $3\times 3$ determinants.

We form new ETFs from the duals of projective planes containing hyperovals, as well as special affine planes that they contain.
The requisite designs are produced by the following:
\begin{lemma}
\label{lemma.dual hyperoval decomposition}
If a projective plane of order $q$ contains a hyperoval then its dual has an incidence matrix of the form
\begin{equation}
\label{equation.dual hyperoval decomposition of projective}
\bfY=\left[\begin{array}{cc}\bfY_{1,1}&\bfY_{1,2}\\\bfzero&\bfY_{2,2}\end{array}\right]
\end{equation}
where $\bfY_{1,1}$ and $\bfY_{2,2}^\rmT$ are the incidence matrices of a $\BIBD(\frac12q(q-1),\frac12q,1)$ and $\BIBD(q+2,2,1)$,  respectively.
Moreover, removing any one of the last $q+2$ rows of $\bfY$ along with the $q+1$ columns it indicates produces an affine plane of order $q$ with an incidence matrix of the form
\begin{equation}
\label{equation.dual hyperoval decomposition of affine}
\bfZ=\left[\begin{array}{cc}\bfY_{1,1}&\bfZ_{1,2}\\\bfzero&\bfZ_{2,2}\end{array}\right]
\end{equation}
where $\bfZ_{2,2}^\rmT$ is the incidence matrix of a $\BIBD(q+1,2,1)$.
\end{lemma}

\begin{proof}
For~\eqref{equation.dual hyperoval decomposition of projective},
take the transpose of the decomposition given in Lemma~\ref{lemma.hyperoval decomposition} and then permute rows and columns.
For the remaining conclusions, recall from Section~II that removing any single block (row) of $\bfY$ along with the $q+1$ vertices (columns) it contains produces an affine plane of order $q$.
Choosing one of the last $q+2$ blocks in particular removes one row and $q+2$ columns from $\bfY_{2,2}$, removes one row from $\bfzero$, removes $q+1$ columns from $\bfY_{1,2}$ and leaves $\bfY_{1,1}$ untouched,
resulting in a matrix of form~\eqref{equation.dual hyperoval decomposition of affine}.
Moreover, since the columns of $\bfY_{2,2}$ indicate all pairs of $q+2$ rows,
the columns of $\bfZ_{2,2}$ indicate all pairs of the remaining $q+1$ rows,
meaning $\bfZ_{2,2}^\rmT$ is the incidence matrix of a $\BIBD(q+1,2,1)$.
\end{proof}

For example, the projective plane of order $2$ given in~\eqref{equation.6x4 and 7x7 BIBDs} contains a hyperoval since it has form \eqref{equation.hyperoval decomposition}.
Taking its transpose and then permuting rows and columns gives a $7\times 7$ incidence matrix of a projective plane that has form~\eqref{equation.dual hyperoval decomposition of projective}:
\begin{equation}
\setlength{\arraycolsep}{2pt}
\label{equation.6x4 and 7x7 hyperoval BIBDs}
\left[\begin{array}{r|rrrrrr}
1&0&0&1&1&0&0\\
1&0&1&0&0&1&0\\
1&1&0&0&0&0&1\\
\hline
0&1&1&1&0&0&0\\
0&1&0&0&1&1&0\\
0&0&1&0&1&0&1\\
0&0&0&1&0&1&1
\end{array}\right],
\quad
\left[\begin{array}{r|rrr}
1&1&0&0\\
1&0&1&0\\
1&0&0&1\\
\hline
0&1&1&0\\
0&1&0&1\\
0&0&1&1
\end{array}\right].
\end{equation}
The second matrix here is an example of an incidence matrix of an affine plane that has form~\eqref{equation.dual hyperoval decomposition of affine}.
It is obtained by removing the fourth row and columns $2$, $3$ and $4$ from the first matrix.

We form Steiner-like ETFs by using such BIBDs to embed the rows of two distinct (possibly complex) Hadamard matrices,
one of size $q+2$ and the other of size $q$.
As with Steiner ETFs, we remove a row from a Hadamard matrix of size $q+2$ to obtain a $(q+1)\times(q+2)$ matrix $\bfS$ whose columns $\set{\bfs_l}_{l=1}^{q+2}$ form a unimodular simplex for $\bbF^{q+1}$.
See~\eqref{equation.3x4 unimodular simplex} for an example of $\bfS$ for $q=2$.
We also take the negative of the last row of a $q\times q$ Hadamard matrix and attach it to its bottom to produce a $(q+1)\times q$ matrix $\bfC$.
For example, when $q=2$,
\begin{equation}
\setlength{\arraycolsep}{2pt}
\label{equation.3x2 unimodular cosimplex}
\bfC
=\left[\begin{array}{rr}
+&+\\
+&-\\
-&+
\end{array}\right].
\end{equation}
The columns of such a matrix $\bfC$ have the following properties:

\begin{definition}
\label{definition.cosimplex}
For any $r\geq 3$, a corresponding \textit{unimodular cosimplex} is a sequence of vectors $\set{\bfc_l}_{l=1}^{r-1}$ in $\bbF^r$ with the property that each $\bfc_l$ has unimodular entries, the last two entries of any $\bfc_l$ sum to zero, and $\abs{\ip{\bfc_l}{\bfc_{l'}}}=1$ for all $l\neq l'$.
\end{definition}

To form new ETFs,
let $\bfZ$ be the incidence matrix of an affine plane of order $q$ of form~\eqref{equation.dual hyperoval decomposition of affine},
use the first \smash{$\frac12q(q-1)$} of the operators $\set{\bfE_j}_{j=1}^{q^2}$ (cf.\ Definition~\ref{definition.embedding}) to embed $\bfS$,
and use the remaining operators to embed $\bfC$.
For example, using the first column of the $6\times 4$ affine plane in~\eqref{equation.6x4 and 7x7 hyperoval BIBDs} to embed~\eqref{equation.3x4 unimodular simplex} and using the remaining columns to embed~\eqref{equation.3x2 unimodular cosimplex}  gives
$\left[\begin{array}{rrrr}\bfE_1\bfS&\bfE_2\bfC&\bfE_3\bfC&\bfE_4\bfC\end{array}\right]$, namely
\begin{equation}
\setlength{\arraycolsep}{2pt}
\label{equation.5x10 ETF from hyperoval}
\left[\begin{array}{rrrr|rrrrrr}
+&-&-&+&+&+&0&0&0&0\\
+&+&-&-&0&0&+&+&0&0\\
+&-&+&-&0&0&0&0&+&+\\
\hline
0&0&0&0&+&-&+&-&0&0\\
0&0&0&0&-&+&0&0&+&-\\
0&0&0&0&0&0&-&+&-&+
\end{array}\right].
\end{equation}
Note that by inspection, the $10$ columns of this matrix are equiangular with coherence $\tfrac13$.
However, these vectors do not form an ETF for $\bbR^6$: the rows of~\eqref{equation.5x10 ETF from hyperoval} are not orthogonal,
and moreover letting $n=10$ and $d=6$ in~\eqref{equation.Welch bound} does not yield $\frac13$.
Rather, by the final statement of Lemma~\ref{lemma.tight for subspace}, they form an ETF for a $5$-dimensional orthogonal complement of $(0,0,0,1,1,1)$.

When applied to the affine plane produced by Lemma~\ref{lemma.dual hyperoval decomposition} from the canonical hyperoval and projective plane of order $n=4$, this same approach yields the $76$-vector ETF for a $19$-dimensional subspace of $\bbC^{20}$, cf.\ Figure~\ref{figure.19x76}.
It can also be applied to projective planes of form~\eqref{equation.dual hyperoval decomposition of projective}.
For example, using the $7\times 7$ projective plane in~\eqref{equation.6x4 and 7x7 hyperoval BIBDs} to embed~\eqref{equation.3x4 unimodular simplex} and~\eqref{equation.3x2 unimodular cosimplex} produces
$\left[\begin{array}{rrrrrrr}\bfE_1\bfS&\bfE_2\bfC&\bfE_3\bfC&\bfE_4\bfC&\bfE_5\bfC&\bfE_6\bfC&\bfE_7\bfC\end{array}\right]$,
a $16$-vector ETF for a $6$-dimensional orthogonal complement of $(0,0,0,1,1,1,1)$ in $\bbR^7$.
We now formally prove these facts.

\begin{figure*}
\begin{equation*}
\bfPhi=\left[\begin{array}{cccccc|cccccccccc}
a&a& & & & & & & & & &f& & &f& \\
 & &a&a& & & & & & &f& & & & &f\\
 & & & &a&a& & & & & & &f&f& & \\
b& &b& & & & & & &f& & & &g& & \\
 &b& & &b& & &f& & & & & & & &g\\
 & & &b& &b& & &f& & & & & &g& \\
c& & & & &c&f& & & & & & & & &h\\
 &c&c& & & & & &g& & & &g& & & \\
 & & &c&c& & & & &g& &g& & & & \\
d& & &d& & & &g& & & & &h& & & \\
 &d& & & &d& & & &h&g& & & & & \\
 & &d& &d& &g& & & & & & & &h& \\
e& & & &e& & & &h& &h& & & & & \\
 &e& &e& & &h& & & & & & &h& & \\
 & &e& & &e& &h& & & &h& & & & \\
\hline
 & & & & & &i&i&i&i& & & & & & \\
 & & & & & &j& & & &i&i&i& & & \\
 & & & & & & &j& & &j& & &i&i& \\
 & & & & & & & &j& & &j& &j& &i\\
 & & & & & & & & &j& & &j& &j&j
\end{array}\right],
\setlength{\arraycolsep}{2pt}
\
\begin{array}{rcl}
\omega&=&\exp(\frac{2\pi\rmi}6),
\\ \ \\
\left[\begin{array}{r}a\\b\\c\\d\\e\end{array}\right]
&=&\left[\begin{array}{cccccc}
1&\omega^2&\omega^4&       1&\omega^2&\omega^4\\
1&\omega^4&\omega^2&       1&\omega^4&\omega^2\\
1&       1&       1&\omega^3&\omega^3&\omega^3\\
1&\omega^2&\omega^4&\omega^3&\omega^5&\omega^1\\
1&\omega^4&\omega^2&\omega^3&\omega^1&\omega^5
\end{array}\right],
\\\ \\
\left[\begin{array}{r}f\\g\\h\\i\\j\end{array}\right]
&=&\left[\begin{array}{rrrr}
 1& 1& 1& 1\\
 1&-1& 1&-1\\
 1& 1&-1&-1\\
 1&-1&-1& 1\\
-1& 1& 1&-1
\end{array}\right].
\end{array}
\end{equation*}
\caption[LoF entry]{
\label{figure.19x76}
A complex ETF of $n=76$ vectors for the $19$-dimensional subspace of $\bbC^{20}$ that consists of vectors whose last $5$ entries sum to zero. It was recently shown that no real ETF with parameters $n=76$, $d=19$ exists~\cite{AzarijaM15,Yu15}.
To increase readability, we denote the rows of a $5\times 6$ and $5\times 4$ unimodular simplex and cosimplex by the first $10$ letters of the alphabet.
Blank entries denote rows of zeros of the appropriate size.
This ETF is formed in the manner of Theorem~\ref{theorem.main result}:
a hyperoval in a projective plane of order $4$ decomposes a certain affine plane, cf.\ Lemma~\ref{lemma.dual hyperoval decomposition};
this affine plane is then used to embed $6$ copies of the simplex and $14$ copies of the cosimplex into $\bbC^{20}$.
The first $36$ vectors form a known Steiner ETF for $\bbC^{15}$ arising from a Denniston design~\cite{FickusMT12}.
The last $40$ of these vectors are real.

\quad
In detail, we form a projective plane of order $q=2^2=4$ by identifying $\bbF_q^3$ with the additive group of $\bbF_{q^3}=\bbF_{64}=\bbZ_2(\alpha)$,
where $\alpha$ is a root of the primitive polynomial \smash{$\beta^6+\beta+1$} over $\bbZ_2$~\cite{HansenM92}.
For the sake of simplicity, we performed field calculations in MATLAB using an isomorphism between $\bbF_{64}$ and certain $6\times 6$ matrices over $\bbZ_2$~\cite{Wardlaw94}.
The set of vertices in our projective plane is  \smash{$\calV=\bbF_{64}^\times/\bbF_{4}^\times=\langle\alpha\rangle/\langle\alpha^{21}\rangle\cong\bbZ_{21}$},
and all blocks correspond to cyclic shifts of the Singer difference set
$\set{i\in\bbZ_{21}: 0=\tr(\alpha^i)=\alpha^i+\alpha^{4i}+\alpha^{16i}}=\set{3,6,7,12,14}$.
Here, the canonical hyperoval $\set{0,1,2,3,5,14}$ is obtained by taking the logarithm base $\alpha$ of the vertices $\set{t+t^2\alpha+\alpha^2: t\in\bbF_4}\cup\set{1}\cup\set{\alpha}$ given in~\eqref{equation.hyperoval construction}.
This hyperoval provides a $21\times 21$ incidence matrix of form~\eqref{equation.hyperoval decomposition}.
Permuting its dual (transpose) gives a matrix of form~\eqref{equation.dual hyperoval decomposition of projective}.
Removing a ``hyperoval" row and its corresponding columns produces the $20\times 16$ affine plane of form~\eqref{equation.dual hyperoval decomposition of affine} that we used above.
}
\end{figure*}

\begin{theorem}
\label{theorem.main result}
For a projective plane of order $q$ that contains a hyperoval,
let \smash{$\set{\bfE_j}_{j=1}^{q^2}$} be embeddings arising from an affine plane of form~\eqref{equation.dual hyperoval decomposition of affine},
cf.\ Definition~\ref{definition.embedding}.
Let $\set{\bfs_l}_{l=1}^{q+2}$ and $\set{\bfc_l}_{l=1}^{q}$ be a unimodular simplex and cosimplex for $\bbF^{q+1}$, respectively,
cf.\ Definition~\ref{definition.cosimplex}.
Then
\begin{equation}
\label{equation.main result affine}
\set{\bfE_j\bfs_l}_{l=1,}^{q+2}\,_{j=1}^{\frac12q(q-1)}\cup\set{\bfE_j\bfc_l}_{l=1,}^{q}\,_{j=\frac12q(q-1)+1}^{q^2}
\end{equation}
is a $q(q^2+q-1)$-vector ETF for the $(q^2+q-1)$-dimensional subspace of $\bbF^{q(q+1)}$ that consists of those vectors whose last \smash{$q+1$} entries sum to zero.

Moreover, if we instead let \smash{$\set{\bfE_j}_{j=1}^{q^2+q+1}$} be embeddings arising from a projective plane of form~\eqref{equation.dual hyperoval decomposition of projective}, then
\begin{equation}
\label{equation.main result projective}
\set{\bfE_j\bfs_l}_{l=1,}^{q+2}\,_{j=1}^{\frac12q(q-1)}\cup\set{\bfE_j\bfc_l}_{l=1,}^{q}\,_{j=\frac12q(q-1)+1}^{q^2+q+1}
\end{equation}
is a $q^2(q+2)$-vector ETF for the $q(q+1)$-dimensional subspace of \smash{$\bbF^{q^2+q+1}$} of vectors whose last \smash{$q+2$} entries sum to zero.
\end{theorem}

\begin{proof}
We prove this result in the affine case; the proof in the projective case is similar.
When $n=q(q^2+q-1)$ and $d=q^2+q-1$, the Welch bound~\eqref{equation.Welch bound} is $\frac1{q+1}$.
Moreover, by definition, each embedding $\bfE_j$  maps an orthonormal basis of its domain to one of its range.
This means each $\bfE_j$ is an isometry, namely $\bfE_j^*\bfE_j^{}=\bfI$.
Since each $\bfs_l$ and $\bfc_l$ has $q+1$ entries, all unimodular,
this implies $\norm{\bfE_j\bfs_l}^2=\norm{\bfs_l}^2=q+1$ and similarly $\norm{\bfE_j\bfc_l}^2=q+1$.
As such, to prove the first claim using Lemma~\ref{lemma.tight for subspace}, it thus suffices to show that (i) each of the vectors in~\eqref{equation.main result affine} lies in a hyperplane of $\bbF^{q(q+1)}$ and that (ii) the inner product of any two of these vectors has unit modulus.

For (i), we claim all vectors in~\eqref{equation.main result affine} lie in $\bfd^\perp$ where
\begin{equation*}
\setlength{\arraycolsep}{1pt}
\bfd\in\bbF^{q(q+1)},
\quad
\bfd(i)=\left\{\begin{array}{ccl}0,&\ 1&\leq i\leq q^2-1,\\1,&\ q^2&\leq i\leq q(q+1).\end{array}\right.
\end{equation*}
Indeed, for any $j$ with $1\leq j\leq\frac12q(q-1)$, the form~\eqref{equation.dual hyperoval decomposition of affine} of our affine plane gives $(\bfE_j\bfs_l)(i)=0$ for all $l$ whenever $q^2\leq i\leq q(q+1)$, and so $\ip{\bfd}{\bfE_j\bfs_l}=0$.
In the remaining case where $\frac12q(q-1)+1\leq j\leq q^2+q+1$, the form of our affine plane along with a property of the cosimplex gives
\begin{equation*}
\ip{\bfd}{\bfE_j\bfc_l}
=\sum_{i=q^2}^{q(q+1)}(\bfE_j\bfc_l)(i)
=\bfc_l(q)+\bfc_l(q+1)
=0.
\end{equation*}

For (ii), since each $\bfE_j$ is an isometry and $\set{\bfs_l}_{l=1}^{q+2}$ is a unimodular simplex,
$\abs{\ip{\bfE_j\bfs_l}{\bfE_j\bfs_{l'}}}=\abs{\ip{\bfs_l}{\bfs_{l'}}}=1$ whenever $l\neq l'$.
Similarly, since $\set{\bfc_l}_{l=1}^{q+2}$ is a unimodular cosimplex,
$\abs{\ip{\bfE_j\bfc_l}{\bfE_j\bfc_{l'}}}=\abs{\ip{\bfc_l}{\bfc_{l'}}}=1$ whenever $l\neq l'$.
A different argument is required for inner products of the form
\begin{equation}
\label{equation.proof of main result 1}
\ip{\bfE_j\bfs_l}{\bfE_{j'}\bfs_{l'}},\quad
\ip{\bfE_j\bfs_l}{\bfE_{j'}\bfc_{l'}},\quad
\ip{\bfE_j\bfc_l}{\bfE_{j'}\bfc_{l'}},
\end{equation}
when $j\neq n'$.
There, the fact that any two distinct blocks in our BIBD have exactly one vertex in common implies that exactly one column of $\bfE_j$ appears as a column of $\bfE_{j'}$ while all other columns have disjoint support.
This implies $\bfE_j^*\bfE_{j'}^{}$ is the outer product $\bfdelta_i^{}\bfdelta_{i'}^*$ of two standard basis elements of $\bbF^r$,
cf.\ Lemma~2.1 of~\cite{FickusJMP16}.
In particular,
\begin{equation*}
\ip{\bfE_j\bfs_l}{\bfE_{j'}\bfs_{l'}}
=\ip{\bfs_l}{\bfE_j^*\bfE_{j'}^{}\bfs_{l'}}
=\ip{\bfs_l}{\bfdelta_i\bfdelta_{i'}^*\bfs_{l'}}
=\overline{\bfs_l(i)}\bfs_{l'}(i').
\end{equation*}
Similarly, the other two inner products in~\eqref{equation.proof of main result 1} are $\overline{\bfs_l(i)}\bfc_{l'}(i')$ and $\overline{\bfc_l(i)}\bfc_{l'}(i')$.
Since the entries of every $\bfs_l$ and $\bfc_l$ are unimodular, these inner products have unit modulus.
\end{proof}

We have some remarks about this result.
First, recall that all known constructions of hyperovals lie in projective planes of order $q=2^e$ for some $e\geq1$.
Further recall that for any such $q$, the canonical projective plane of order $q$ contains the hyperoval~\eqref{equation.hyperoval construction}.
In this case, we construct the requisite unimodular simplex and cosimplex from Hadamard matrices of size $q+2=2^e+2$ and $q=2^e$, respectively.
The second of these two Hadamard matrices can always be chosen to be real,
obtained for example by taking the Kronecker product of the canonical $2\times 2$ Hadamard matrix with itself $e$ times.
However, $q+2=2^e+2$ is not divisible by $4$ except when $e=1$,
meaning our first Hadamard matrix is necessarily complex except when $q=2$.
That is, we obtain our (complex) unimodular simplex by removing a row from a discrete Fourier transform of size $q+2$, for example.
Together, these facts imply:
\begin{corollary}
For any $e\geq 1$, applying Theorem~\ref{theorem.main result} to the canonical projective plane of order $q=2^e$ gives an $n$-vector ETF for $\bbC^d$ where $n$ and $d$ are given by~\eqref{equation.new ETF parameters}.
This ETF is real when $e=1$, and is otherwise complex.
\end{corollary}

A second remark: while the ETF~\eqref{equation.main result affine} arising from an affine plane of order $q=2^e$ is new,
the ETF~\eqref{equation.main result projective} arising from a projective plane may not be.
To be precise, two $n$-vector ETFs for $\bbH_d$ are \textit{equivalent} if their synthesis operators satisfy
\begin{equation*}
\bfPhi_1=\bfU\bfPhi_2\bfD\bfP
\end{equation*}
for some unitary operator $\bfU$ on $\bbH_d$ and some $n\times n$ matrices $\bfD$ and $\bfP$ where $\bfD$ is diagonal with unimodular diagonal entries and $\bfP$ is a permutation.
Two ETFs that are constructed in different ways may, in fact, be equivalent.
For example, in~\cite{JasperMF14}, it is shown that every harmonic ETF arising from a McFarland difference set~\cite{DingF07} is equivalent to a special type of Steiner ETF arising from an affine geometry~\cite{FickusMT12}.
The ETF~\eqref{equation.main result projective} has $n=q^2(q+2)$ and $d=q(q+1)$, and so it might be equivalent to Steiner ETFs from affine planes.
We do not know:
determining whether two ETFs are equivalent is similar to determining whether two graphs are isomorphic,
and we leave a deeper investigation of~\eqref{equation.main result projective} for future work.

For any $q=2^e$ where $e>1$, the ETF~\eqref{equation.main result affine} is new~\cite{FickusM15}.
Indeed, strongly regular graphs corresponding to real ETFs with these parameters are not known to exist~\cite{Brouwer07,Brouwer15},
and have been shown to not exist when $q=4$~\cite{AzarijaM15,Yu15}.
Moreover, no ETF with these parameters can be harmonic since
\begin{equation*}
\frac{d(d-1)}{n-1}=q+1-\frac{2q+3}{(q+1)^2}
\end{equation*}
is not an integer;
for harmonic ETFs, this quantity is the number of ways a nonzero element of the group can be written as a difference of members of the difference set.
Also, no ETF with these parameters is a Steiner ETF, since by Theorem~2 of~\cite{FickusMT12}, the corresponding $\BIBD(v,k,1)$ would have  $b=q^2+q-1$ and $r=q+1$ and so
\begin{equation*}
k=\frac{r(r-1)}{r^2-b}=q-1+\frac2{q+2}
\end{equation*}
which is not an integer.

Another remark about Theorem~\ref{theorem.main result}:
note~\eqref{equation.dual hyperoval decomposition of affine} contains the incidence matrix $\bfY_{1,1}$ of a $\BIBD(\frac12q(q-1),\frac12q,1)$,
and is also contained in~\eqref{equation.dual hyperoval decomposition of projective}.
This means the Steiner ETF $\set{\bfE_j\bfs_l}_{l=1,}^{q+2}\,_{j=1}^{\frac12q(q-1)}$ arising from this Denniston design is contained in the ETF~\eqref{equation.main result affine}, which in turn is contained in the ETF~\eqref{equation.main result projective}.
To our knowledge, this is the first time three nontrivial \textit{nested} ETFs have been discovered.
This possibility suggests a new program for discovering ETFs: given an existing ETF, find other ETFs it contains as well as other ETFs that contain it.

A final remark:
note that by Lemma~\ref{lemma.tight for subspace}, the frame operator of~\eqref{equation.main result affine} is necessarily a scalar multiple of a projection onto the subspace of vectors in $\bbF^{q(q+1)}$ whose last $q+1$ entries sum to zero.
In fact, this condition implies the construction of Theorem~\ref{theorem.main result} does not generalize to any BIBDs that are not affine or projective planes.
To be precise, let
\begin{equation}
\label{equation.general form of BIBD}
\bfX=\left[\begin{array}{cc}\bfX_{1,1}&\bfX_{1,2}\\\bfzero&\bfX_{2,2}\end{array}\right]
\end{equation}
be the incidence matrix of any $\BIBD(v,k,1)$ where $\bfX_{1,1}$ and $\bfX_{2,2}^\rmT$ are the incidence matrices of some $\BIBD(v_0,k_0,1)$ and $\BIBD(b_0,2,1)$, respectively.
Here, \eqref{equation.BIBD properties} and~\eqref{equation.BIBD parameters} imply
\begin{equation*}
k_0=k-\frac{k(k+1)}{2r},
\quad
v_0=r(k_0-1)+1,
\quad
b_0=k+1.
\end{equation*}
Use the first $v_0$ columns of $\bfX$ to embed an $r\times(r+1)$ unimodular simplex $\bfS$, and use the remaining columns to embed an $r\times(r-1)$ unimodular cosimplex $\bfC$ \`{a} la Theorem~\ref{theorem.main result}.
Let $\bfPhi$ be the resulting $b\times[v(r-1)+2v_0]$ synthesis operator.
The columns of $\bfPhi$ are equiangular; we want to know when they form a tight frame for their span.
Since $\bfS\bfS^*=(r+1)\bfI$,
\begin{equation*}
\bfC\bfC^*=(r-1)\left[\begin{array}{crr}\bfI&\bfzero&\bfzero\\\bfzero&1&-1\\\bfzero&-1&1\end{array}\right],
\end{equation*}
the structure of $\bfX$ implies the frame operator is
\begin{equation*}
\bfPhi\bfPhi^*=\left[\begin{array}{cc}
k(r+1-\frac{k+1}r)\bfI&\bfzero\\
\bfzero&(r-1)[(k+1)\bfI-\bfJ]
\end{array}\right].
\end{equation*}
Computing the spectrum of $\bfPhi\bfPhi^*$,
we see it is a multiple of an orthogonal projection if and only if $k(r+1-\frac{k+1}r)=(k+1)(r-1)$,
or equivalently, $0=(r-k)[r-(k+1)]$.
By Lemma~\ref{lemma.tight for subspace}, we thus see these equiangular vectors form an ETF for their span if and only if either $r=k$ or $r=k+1$, namely when the underlying BIBD is either a projective plane or affine plane of form~\eqref{equation.general form of BIBD}.

\section{Flat ETFs}

A matrix $\bfPhi$ is \textit{flat} if all of its entries have constant modulus.
Flat waveforms are often used in real-world applications such as radar and wireless communication since they allow a transmitted waveform to have the maximal amount of energy subject to the transmitter's power limit.
That is, mathematically speaking, flat vectors provide the largest possible ratio of $2$-norm to $\infty$-norm.
In light of this, it is natural to investigate ETFs that have flat synthesis operators.

Previous work on this topic has focused on flat $m\times n$ matrices whose columns form an ETF for $\bbF^m$.
In particular, the synthesis operator of a harmonic ETFs is flat, being a submatrix of a character table (discrete Fourier transform) of a finite abelian group.
Some of these ETFs have recently been used to construct minimally-coherent vectors in a regime where no ETF can exist~\cite{BodmannH15}.
Steiner ETFs, on the other hand, are very sparse.
Nevertheless, whenever the underlying BIBD is resolvable, one can rotate the synthesis operator of a Steiner ETF so as to make it flat~\cite{JasperMF14}.

We begin this section by generalizing the method of~\cite{JasperMF14} so as to apply it to the ETF in~\eqref{equation.main result affine}.
This produces a flat $m\times n$ matrix $\bfPhi$ whose columns form an ETF for a proper subspace $\bbH_d$ of $\bbF^m$.
We then consider such ETFs in general, discussing a connection between them and \textit{supersaturated designs},
as well as how some of them imply the existence of other ETFs.

From Section~II, recall that parallelism is an equivalence relation on the blocks of an affine plane,
and that this implies its blocks can be arranged as $q+1$ groupings of $q$ disjoint blocks.
In particular, two rows of~\eqref{equation.dual hyperoval decomposition of affine} lie in the same parallel class precisely when they have disjoint support.
Let $\bfP$ be a permutation matrix that rearranges the rows of~\eqref{equation.dual hyperoval decomposition of affine} into these parallel classes.
For example, for the affine plane of order $2$ in~\eqref{equation.6x4 and 7x7 hyperoval BIBDs},
take $\bfP$ to be the $6\times 6$ permutation matrix such that
\begin{equation}
\label{equation.6x6 permutation}
\setlength{\arraycolsep}{2pt}
\bfP
\left[\begin{array}{rrrr}
1&1&0&0\\
1&0&1&0\\
1&0&0&1\\
0&1&1&0\\
0&1&0&1\\
0&0&1&1
\end{array}\right]
=\left[\begin{array}{rrrr}
0&0&1&1\\
1&1&0&0\\
\hline
0&1&0&1\\
1&0&1&0\\
\hline
0&1&1&0\\
1&0&0&1
\end{array}\right].
\end{equation}
Since a permutation is unitary, applying it to the synthesis operator $\bfPhi$ of the ETF~\eqref{equation.main result affine} that arises from~\eqref{equation.dual hyperoval decomposition of affine} yields another ETF.
For example, recall the columns of the $6\times 10$ matrix $\bfPhi$ given in~\eqref{equation.5x10 ETF from hyperoval} form an ETF for the orthogonal complement of $(0,0,0,1,1,1)$.
Taking $\bfP$ from~\eqref{equation.6x6 permutation}, the columns of
\begin{equation}
\setlength{\arraycolsep}{2pt}
\label{equation.6x10 permuted}
\bfP\bfPhi
=\left[\begin{array}{rrrrrrrrrr}
0&0&0&0&0&0&-&+&-&+\\
+&-&-&+&+&+&0&0&0&0\\
\hline
0&0&0&0&-&+&0&0&+&-\\
+&+&-&-&0&0&+&+&0&0\\
\hline
0&0&0&0&+&-&+&-&0&0\\
+&-&+&-&0&0&0&0&+&+
\end{array}\right],
\end{equation}
thus form a $10$-vector ETF for the orthogonal complement of $(1,0,1,0,1,0)$.
Note here that we have chosen the first member of every parallel class from the bottom $q+1$ rows of~\eqref{equation.dual hyperoval decomposition of affine}.
This is always possible since $\bfZ_{2,2}^\rmT$ is the incidence matrix of a $\BIBD(q+1,2,1)$, meaning none of those $q+1$ rows have disjoint support.
Doing so ensures that the columns of $\bfP\bfPhi$ are orthogonal to $\bfone\otimes\bfdelta_1$ in general
(the concatenation of $q+1$ copies of the first standard basis element in $\bbF^q$).

With $\bfP\bfPhi$ in this form, we now multiply it on the left by a block-diagonal matrix $\bfI\otimes\bfH$ whose diagonal blocks are all some Hadamard matrix $\bfH$ of size $q$ whose first column is $\bfone$.
Since $\bfH$ is a scalar multiple of a unitary operator, the columns of $(\bfI\otimes\bfH)\bfP\bfPhi$ form an ETF for the orthogonal complement of $(\bfI\otimes\bfH)(\bfone\otimes\bfdelta_1)=\bfone\otimes\bfone=\bfone$.
For example, multiplying~\eqref{equation.6x10 permuted} by $\bfI\otimes\bfH$ where $\bfH$ is the canonical $2\times 2$ Hadamard matrix gives a $10$-vector ETF for the orthogonal complement of $(1,1,1,1,1,1)$,
namely for the space of vectors in $\bbR^6$ whose entries sum to zero:
\begin{equation}
\setlength{\arraycolsep}{2pt}
\label{equation.6x10 permuted and rotated}
\left[\begin{array}{cccccccccc}
+&-&-&+&+&+&-&+&-&+\\
-&+&+&-&-&-&-&+&-&+\\
+&+&-&-&-&+&+&+&+&-\\
-&-&+&+&-&+&-&-&+&-\\
+&-&+&-&+&-&+&-&+&+\\
-&+&-&+&+&-&+&-&-&-
\end{array}\right].
\end{equation}
As seen in this example, $(\bfI\otimes\bfH)\bfP\bfPhi$ is flat since each column of $\bfP\bfPhi$ has only one index of support in each parallel class.

Note here that whenever $q=2^e$ for some $e\geq1$ (the only case where projective planes that contain hyperovals are known to exist),
both the requisite cosimplex and Hadamard matrix can be chosen to be real,
implying the last $\frac12q^2(q+1)$ vectors in the ETF are $\set{\pm1}$-valued and also have the property that their entries sum to zero.
For example, permuting the rows of the $20\times 76$ in Figure~\ref{figure.19x76} and then multiplying it by $5$ copies of the canonical $4\times 4$ Hadamard matrix gives a flat $20\times 76$ matrix whose columns form a $76$-vector ETF for the space of vectors in $\bbC^{20}$ whose entries sum to zero.
The last $40$ of these vectors are $\set{\pm1}$-valued while the first $36$ take values from the sixth roots of unity.
In summary:
\begin{theorem}
\label{theorem.flat hyperoval ETFs}
For a projective plane of order $q$ that contains a hyperoval,
let $\bfPhi$ be the synthesis operator of the ETF~\eqref{equation.main result affine}.
There exists a permutation matrix $\bfP$ and a Hadamard matrix $\bfH$ of size $q$ such that $(\bfI\otimes\bfH)\bfP\bfPhi$ is flat (unimodular) and its columns form an ETF for the space of vectors in $\bbF^{q(q+1)}$ whose entries sum to zero.

In particular, when $q=2^e$, $e\geq1$, choosing the simplex, cosimplex and Hadamard matrix appropriately yields an ETF of this type where each vector's entries are $(q+2)$th roots of unity, with the last $\frac12q^2(q+1)$ vectors being $\set{\pm1}$-valued.
\end{theorem}

\subsection{Flat ETFs and supersaturated designs}

The flat ETFs given in Theorem~\ref{theorem.flat hyperoval ETFs} are closely related to supersaturated designs in the field of statistics known as design of experiments.
There, one seeks $\set{\pm1}$-valued $m\times n$ matrices $\bfPhi$ whose columns $\set{\bfphi_i}_{i=1}^{n}$ are orthogonal to $\bfone$ and are also maximally orthogonal in some sense.
To relate that theory to our work here, recall the proof of Lemma~\ref{lemma.tight for subspace}:
for any $\set{\bfphi_i}_{i=1}^{n}$ in a $d$-dimensional subspace $\bbH_d$ of $\bbF^m$ with $\norm{\bfphi_i}^2=r$ for all $i$, rewriting \eqref{equation.proof of Welch bound} gives:
\begin{equation}
\label{equation.DOE Welch}
\tfrac{r^2(n-d)}{d(n-1)}
\leq\tfrac1{n(n-1)}\sum_{\substack{i,j=1\\i\neq j}}^{n}\abs{\ip{\bfphi_i}{\bfphi_j}}^2\\
\leq\max_{i\neq j}\abs{\ip{\bfphi_i}{\bfphi_j}}^2,
\end{equation}
where the first and second inequalities hold with equality precisely when $\set{\bfphi_i}_{i=1}^{n}$ is a tight frame for $\bbH_d$ and equiangular, respectively.
In supersaturated designs, each $\bfphi_i$ is restricted to be $\set{\pm1}$-valued and lie in the orthogonal complement of $\bfone\in\bbR^m$.
Under these assumptions, the first half of~\eqref{equation.DOE Welch} becomes
\begin{equation*}
\frac{m^2(n-m+1)}{(m-1)(n-1)}\leq\rmE(s^2):=\frac1{n(n-1)}\sum_{i\neq j}^{n}\abs{\ip{\bfphi_i}{\bfphi_j}}^2.
\end{equation*}
In the supersaturated design literature, this inequality is well-known~\cite{Nguyen96,TangW97},
and $\set{\pm1}$-valued tight frames for $\bfone^\perp$ are called \textit{$\rmE(s^2)$-optimal designs}.
In light of the second half of~\eqref{equation.DOE Welch}, we pose the following problem:

\begin{problem}
\label{problem.DOE}
For what values of $m,n$ does there exist a sequence $\set{\bfphi_i}_{i=1}^{n}$ of $\set{\pm1}$-valued vectors that forms an ETF for the orthogonal complement of $\bfone$ in $\bbR^m$, and so has
\begin{equation*}
\abs{\ip{\bfphi_i}{\bfphi_j}}
=m\biggl[\frac{(n-m+1)}{(m-1)(n-1)}\biggr]^{\frac12},
\quad\forall i\neq j?
\end{equation*}
\end{problem}

Such ETFs are optimal in \textit{minimax} sense, cf.~\cite{RyanB07}.
It is unclear when they exist:
reviewing both the ETF and supersaturated design literature, the only example of such an ETF we could find is when $m=6$ and $n=10$, namely~\eqref{equation.6x10 permuted and rotated} or the equivalent ETF~\eqref{equation.5x10 DOE}.
In the case where $m=q(q+1)$, $n=q(q^2+q-1)$ for some $q=2^e$, $e>1$, the ETFs of Theorem~\ref{theorem.flat hyperoval ETFs} almost work.
The only issue is that the entries of their first $\frac12q(q^2+q-2)$ vectors are $(q+2)$th roots of unity, not necessarily $\pm1$.
As a partial solution to Problem~\ref{problem.DOE}, we have the following necessary condition on $m$ and $n$:
\begin{theorem}
\label{theorem.DOE necessary}
If $m>2$ and there are $n$ vectors in $\bbR^m$ with entries in $\set{\pm1}$ that form an ETF for $\bfone^\perp$, then there exists an even integer $q\geq2$ such that
\begin{equation}
\label{equation.DOE necessary}
m=q(q+1),
\quad
n=q(q^2+q-1).
\end{equation}
\end{theorem}

\begin{proof}
We exploit known integrality conditions on the existence of real ETFs:
if $1<d<n-1$ and $n\neq 2d$, and there exists an $n$-vector ETF for a $d$-dimensional real Hilbert space,
then both
\begin{equation}
\label{equation.proof of DOE necessary 1}
\bigg[\frac{d(n-1)}{n-d}\biggr]^{\frac12},
\quad
\bigg[\frac{(n-d)(n-1)}{d}\biggr]^{\frac12},
\end{equation}
are odd integers~\cite{SustikTDH07}.
These quantities are the reciprocals of the Welch bounds for the ETF and its Naimark complements, respectively.
Here, our Hilbert space $\bfone^\perp\subset\bbR^m$ has dimension $d=m-1>1$.
Since our vectors are $\set{\pm1}$-valued, their inner products are integers.
Since they are also orthogonal to $\bfone$, $m=d+1$ is necessarily even.
Since they form an ETF for $\bfone^\perp$, they achieve the Welch bound~\eqref{equation.Welch bound}, meaning the absolute value of their inner products is
\begin{equation}
\label{equation.proof of DOE necessary 2}
(d+1)\bigg[\frac{n-d}{d(n-1)}\biggr]^{\frac12}.
\end{equation}
Together, these facts imply \eqref{equation.proof of DOE necessary 2} is an integer.
This immediately rules out $d=n-1$, since in this case~\eqref{equation.proof of DOE necessary 2} becomes $1+\frac1d$.
Moreover, in the case $n=2d$, a polynomial long division reveals the square of~\eqref{equation.proof of DOE necessary 2} to be
\begin{equation*}
\frac{(d+1)^2}{2d-1}
=\frac14\biggl(2d+5+\frac 9{2d-1}\biggr).
\end{equation*}
Since $d$ is odd, this is only an integer when $d=5$,
in which case $(m,n)=(6,10)$ is of form~\eqref{equation.DOE necessary} with $q=2$;
an example of such an ETF is given in~\eqref{equation.5x10 DOE}.
Knowing $1<d<n-1$ and having handled the case where $n=2d$,
we now assume $n\neq 2d$, implying the numbers in~\eqref{equation.proof of DOE necessary 1} are odd integers.
As such, products of~\eqref{equation.proof of DOE necessary 1} and~\eqref{equation.proof of DOE necessary 2} are integers.
In particular,
\begin{equation*}
(d+1)\bigg[\frac{n-d}{d(n-1)}\biggr]^{\frac12}
\bigg[\frac{(n-d)(n-1)}{d}\biggr]^{\frac12}
=n-d-1+\frac nd
\end{equation*}
is an integer, implying the \textit{redundancy} $q:=\frac nd$ is an integer.
Squaring the integer~\eqref{equation.proof of DOE necessary 2} and writing $n=qd$ gives
\begin{equation*}
\frac{(q-1)(d+1)^2}{(qd-1)}=\frac{r-1}{r^2}\biggl[qd+2q+1+\frac{(q+1)^2}{qd-1}\biggr],
\end{equation*}
implying $qd-1$ divides $(q-1)(q+1)^2$.
As such, there exists a positive integer $j$ such that $(q-1)(q+1)^2=j(qd-1)$.
Modulo $q$, this equation becomes $1\equiv j\bmod q$.
At the same time, the \textit{Gerzon bound}~\cite{LemmensS73} on real ETFs states that $n\leq\frac12d(d+1)$.
Thus, $d^2>n$ and so $d>\frac nd=q$, implying
\begin{equation*}
j
=\frac{(q-1)(q+1)^2}{qd-1}
<\frac{(q-1)(q+1)^2}{q^2-1}
=q+1.
\end{equation*}
Since $0<j<q+1$ and $j\equiv 1\bmod q$, we have $j=1$ and so $(q-1)(q+1)^2=qd-1$.
Solving for $d$ then gives $d=q^2+q-1$.
Noting $m=d+1$ and $n=qd$ gives~\eqref{equation.DOE necessary}.
To show $q$ is even, recall the first quantity in~\eqref{equation.proof of DOE necessary 1} is odd;
since $d=q^2+q-1$ and $n=qd$, this quantity is $q+1$.
\end{proof}

Here, we emphasize that the parameters $m$ and $n$ of the complex flat ETFs produced by Theorem~\ref{theorem.flat hyperoval ETFs} are identical to those given in~\eqref{equation.proof of DOE necessary 1}.
That is, when $q=2^e$, $e\geq1$, defining $m$ and $n$ by~\eqref{equation.proof of DOE necessary 1},
Theorem~\ref{theorem.flat hyperoval ETFs} produces an $m\times n$ complex flat matrix whose columns form an ETF for $\bfone^\perp$.
We also note that in light of Theorem~\ref{theorem.DOE necessary}, Problem~\ref{problem.DOE} is equivalent to:

\begin{problem}
\label{problem.DOE refined}
Given $q$ even,
does there exist a sequence $\set{\bfphi_i}_{i=1}^{q(q^2+q-1)}$ of $\set{\pm1}$-valued vectors that form an ETF for $\bfone^\perp$ in $\bbR^{q(q+1)}$, and so have $\abs{\ip{\bfphi_i}{\bfphi_j}}=q$ for all $i\neq j$?
\end{problem}

The answer to Problem~\ref{problem.DOE refined} is ``yes" when $q=2$, e.g.\ \eqref{equation.5x10 DOE}, and ``no" when $q=4$,
since real ETFs with $(d,n)=(19,76)$ do not exist~\cite{AzarijaM15,Yu15}.
To our knowledge, the $q=6$ case is open.
Here, $(m,n)=(42,246)$ and so computer-assisted searches might be impractical.
Cases like this where $\frac nm=q\equiv 2\bmod 4$ are particularly interesting since they might be related to a fact from~\cite{JasperMF14}:
if there exist $m\times n$ real flat matrices whose columns form an ETF for the entire space $\bbR^m$,
and if these matrices arise from $\BIBD(v,k,1)$ in the manner of~\cite{FickusMT12,JasperMF14},
then $\frac nm\approx k\equiv 2\bmod 4$.

\subsection{Extending flat ETFs to larger ETFs}

We conclude this paper by discussing how we can append vectors to certain flat ETFs to produce other ETFs.
The inspiration for this work was the realization that both the columns and the rows of the matrices discussed in Theorems~\ref{theorem.flat hyperoval ETFs} and~\ref{theorem.DOE necessary} form ETFs for their respective spans, and that the Welch bounds for these ``paired" ETFs are closely related.
To be precise, if $q$ is any positive integer and $\bfPhi$ is any $q(q+1)\times q(q^2+q-1)$ matrix whose columns form an ETF for $\bfone^\perp$ in $\bbF^m$,
then Lemma~\ref{lemma.tight for subspace} implies $\bfPhi\bfPhi^*=a\bfPi$ where $a=q^2(q+1)$ and \smash{$\bfPi=\bfI-\frac1{q(q+1)}\bfJ$} is the orthogonal projection operator onto $\bfone^\perp$.
That is,
\begin{equation*}
\bfPhi\bfPhi^*
=q^2(q+1)\bfI-q\bfJ.
\end{equation*}
This means the columns of $\bfPhi^*$ are equiangular,
with the inner product of any two of them being $-q$.
Remarkably, scaling the Welch bound of its columns by their norms, we find their inner products have the same magnitude:
\begin{equation*}
m\biggl[\frac{n-d}{d(n-1)}\biggr]^{\frac12}
=q(q+1)\biggl[\frac{q-1}{q^3+q^2-q-1}\biggr]^{\frac12}
=q.
\end{equation*}
Even more remarkably, we found that concatenating such $\bfPhi$ with $q\bfI+\frac{\sqrt{q+2}-1}{q-1}\bfJ$ led to even larger ETFs.
For example, for the $10$-vector flat ETF for \smash{$\bfone^\perp\subseteq\bbR^6$} given in~\eqref{equation.5x10 DOE}, we can append $6$ vectors to it to form a $16$-vector ETF for $\bbR^6$:
\begin{equation*}
\setlength{\arraycolsep}{2pt}
\left[\begin{array}{rrrrrrrrrrrrrrrr}
+&+&+&+&+&+&+&+&+&+&\frac73&\frac13&\frac13&\frac13&\frac13&\frac13\smallskip\\
+&+&+&+&-&-&-&-&-&-&\frac13&\frac73&\frac13&\frac13&\frac13&\frac13\smallskip\\
+&-&-&-&+&+&+&-&-&-&\frac13&\frac13&\frac73&\frac13&\frac13&\frac13\smallskip\\
-&+&-&-&+&-&-&+&+&-&\frac13&\frac13&\frac13&\frac73&\frac13&\frac13\smallskip\\
-&-&+&-&-&+&-&+&-&+&\frac13&\frac13&\frac13&\frac13&\frac73&\frac13\smallskip\\
-&-&-&+&-&-&+&-&+&+&\frac13&\frac13&\frac13&\frac13&\frac13&\frac73
\end{array}\right].
\end{equation*}
In general, we do not know whether these ETFs are equivalent to other known families of ETFs with these same parameters,
such as harmonic ETFs arising from certain McFarland difference sets~\cite{DingF07},
Steiner ETFs arising from affine planes~\cite{FickusMT12},
or those instances of~\eqref{equation.main result projective} in Theorem~\ref{theorem.main result} arising from a known hyperoval in a projective plane of order $q=2^e$.
One reason such a construction is possible is that the columns of $\bfPhi$ are orthogonal to those of $\bfJ$, and so only interact with the ``$\bfI$" component of the columns of matrices of the form $f\bfI+g\bfJ$.
Generalizing this construction gives the following result:

\begin{theorem}
\label{theorem.extended ETF}
Suppose $\bfPhi$ is an $m\times n$ non-square matrix with unimodular entries
and that the columns of $\bfPhi$ and $\bfPhi^*$ both form ETFs for their $d$-dimensional spans with
\begin{equation}
\label{equation.extended ETF condition}
\frac1{n}\biggl[\frac{n-d}{d(n-1)}\biggr]^{\frac12}=\frac1{m}\biggl[\frac{m-d}{d(m-1)}\biggr]^{\frac12}.
\end{equation}
Then there exist real scalars $f,g$ such that the columns of
\begin{equation*}
\bfPsi=\left[\begin{array}{cc}\bfPhi&f\bfPhi\bfPhi^*+g\bfI\end{array}\right]
\end{equation*}
form an $(m+n)$-vector ETF for $\bbF^m$, namely
\begin{align}
\nonumber
f&=-\frac{(m+n-1)^{\frac12}}{(m+n)^{\frac12}(m+n-1)^{\frac12}\pm(mn)^{\frac12}},\\
\label{equation.extended ETF parameters}
g&=(m+n)^{\frac12}.
\end{align}
\end{theorem}

\begin{proof}
We begin by expressing $d$ in terms of $m$ and $n$.
Squaring~\eqref{equation.extended ETF condition} and simplifying gives
\begin{equation*}
mn[m(m-1)-n(n-1)]\frac1d=m^2(m-1)-n^2(n-1).
\end{equation*}
Since $\bfPhi$ is not square by assumption, we can divide both sides of this equation by $m-n$ to obtain
\begin{equation*}
mn(m+n-1)\frac1d=(m+n)(m+n-1)-mn.
\end{equation*}
That is, $d=\operatorname{rank}(\bfPhi)=\operatorname{rank}(\bfPhi^*)$ necessarily satisfies
\begin{equation}
\label{equation.proof of extended ETF -1}
\frac1d
=\frac1m+\frac1n-\frac1{m+n-1}.
\end{equation}
This fact in hand, note that since the columns $\set{\bfphi_i}_{i=1}^{n}$ of $\bfPhi$ form an ETF for their span with $\norm{\bfphi_i}^2=m$ for all $i$,
Lemma~\ref{lemma.tight for subspace} gives $\bfPhi\bfPhi^*\bfPhi=a\bfPhi$ where
\begin{equation}
\label{equation.proof of extended ETF 0}
a
=\frac{mn}d
=\frac1m+\frac1n+\frac{mn}{m+n-1}.
\end{equation}
Taking adjoints gives $\bfPhi^*\bfPhi\bfPhi^*=a\bfPhi^*$ as well.
As such, the frame operator of $\bfPsi$ is
\begin{align*}
\bfPsi\bfPsi^*
&=\bfPhi\bfPhi^*+(f\bfPhi\bfPhi^*+g\bfI)(f\bfPhi\bfPhi^*+g\bfI)^*\\
&=(af^2+2fg+1)\bfPhi\bfPhi^*+g^2\bfI.
\end{align*}
Since the $f$ and $g$ given in~\eqref{equation.extended ETF parameters} satisfy
\begin{equation}
\label{equation.proof of extended ETF 1}
af^2+2fg+1=0,
\end{equation}
we have $\bfPsi\bfPsi^*=(m+n)\bfI$.
Thus, the columns of $\bfPsi$ form a tight frame for $\bbF^m$.
Next, the Gram matrix of $\bfPsi$ is
\begin{align}
\bfPsi^*\bfPsi
\nonumber&=\left[\begin{array}{c}\bfPhi^*\\f\bfPhi\bfPhi^*+g\bfI\end{array}\right]\left[\begin{array}{cc}\bfPhi&f\bfPhi\bfPhi^*+g\bfI\end{array}\right]\\
\nonumber&=\left[\begin{array}{cc}\bfPhi^*\bfPhi&(af+g)\bfPhi^*\\(af+g)\bfPhi&(af^2+2fg)\bfPhi\bfPhi^*+g^2\bfI\end{array}\right]\\
\label{equation.proof of extended ETF 2}
&=\left[\begin{array}{cc}\bfPhi^*\bfPhi&(af+g)\bfPhi^*\\(af+g)\bfPhi&g^2\bfI-\bfPhi\bfPhi^*\end{array}\right].
\end{align}
Since $\bfPhi$ is an $m\times n$ matrix with unimodular entries and $g^2=m+n$, all of the diagonal entries of $g^2\bfI-\bfPhi\bfPhi^*$ have value $(m+n)-n=m$, which equals the value of the diagonal entries of $\bfPhi^*\bfPhi$.
Thus, the columns of $\bfPsi$ form an equal norm tight frame for $\bbF^m$.
To show they form an ETF, note we can rewrite~\eqref{equation.extended ETF condition} as
\begin{equation}
\label{equation.proof of extended ETF 3}
m\biggl[\frac{n-d}{d(n-1)}\biggr]^{\frac12}=n\biggl[\frac{m-d}{d(m-1)}\biggr]^{\frac12},
\end{equation}
namely that the off-diagonal entries of $\bfPhi^*\bfPhi$ and $\bfPhi\bfPhi^*$ have the same modulus.
We further note that the values in~\eqref{equation.proof of extended ETF 3} are equal to $\abs{af+g}$.
Indeed, \eqref{equation.proof of extended ETF 1}, \eqref{equation.extended ETF parameters} and~\eqref{equation.proof of extended ETF 0} give
\begin{equation*}
(af+g)^2
=a(af^2+2fg)+g^2
=\frac{mn}{m+n-1},
\end{equation*}
which is the same value obtained by substituting~\eqref{equation.proof of extended ETF -1} into the squares of the quantities in~\eqref{equation.proof of extended ETF 3}, e.g.,
\begin{align*}
m^2\frac{n-d}{d(n-1)}
&=\frac{m^2}{n-1}\Bigl(\frac nm+1-\frac{n}{m+n-1}-1\Bigr)\\
&=\frac{mn}{m+n-1}.
\end{align*}
Recalling $\bfPhi$ has unimodular entries,
this means all the off-diagonal entries of~\eqref{equation.proof of extended ETF 2} have the same modulus, and so the columns of $\bfPsi$ form an ETF for $\bbF^m$.
\end{proof}

In the special case where $d=m-1$, \eqref{equation.proof of extended ETF -1} reduces to having $n^2+(m-1)n-m(m-1)^2=0$, whose only positive solution is $n=\frac12(m-1)[(4m+1)^{\frac12}-1]$.
Here, since $4m+1$ is an odd square, it can be written as
\begin{equation*}
4m+1=(2q+1)^2=4q^2+4q+1
\end{equation*}
for some integer $q$,
meaning $m=q(q+1)$, which in turn gives $n=\frac12(q^2+q-1)[(2q+1)-1]=q(q^2+q-1)$.
For this choice of $m$ and $n$, if there exists an $m\times n$ flat matrix with unimodular entries whose columns form an ETF for $\bfone^\perp$, then Theorem~\ref{theorem.extended ETF} produces an ETF of $q^2(q+2)$ vectors for $\bbF^{q(q+1)}$.
In the real-variable setting, this leads to the following result,
which is an avenue for future research on the nexus of ETFs and supersaturated designs:
\begin{corollary}
If there exists a solution to Problem~\ref{problem.DOE refined} for a given $q$, then there exists a real ETF of $q^2(q+2)$ vectors for $\bbF^{q(q+1)}$.
\end{corollary}
In particular, when $q=6$, if there exists a $\set{\pm1}$-valued matrix of size $42\times246$ whose columns form an ETF for $\bfone^\perp$, then there is an ETF of $288$ vectors in $\bbR^{42}$, resolving an open problem in the theory of strongly regular graphs.

Interestingly, there are choices of $(m,n,d)$ that satisfy~\eqref{equation.extended ETF condition} but do not have $d=m-1$, like $(d,m,n)=(8,10,15)$.
So far, we have been unable to find an ETF with these parameters that meets the hypotheses of Theorem~\ref{theorem.extended ETF}.
We leave a deeper investigation of such ETFs with $d<m-1$ for future research.

We finish with a tantalizing connection between Theorem~\ref{theorem.extended ETF} and difference sets.
To be precise, for a finite abelian group $\calG$ of order $v$,
a $k$-element subset $\calD$ of $\calG$ is a \textit{difference set} of $\calG$ if the cardinality of $\set{(\delta,\varepsilon)\in\calG\times\calG: \gamma=\delta-\varepsilon}$ is constant over all $\gamma\neq0$.
Letting $\bfH$ denote the $v\times v$ character table of $\calG$, it is well-known that the columns of a $k\times v$ submatrix of $\bfH$ form an ETF for $\bbF^k$ if and only if the $k$ rows correspond to a difference set of $\calG$~\cite{XiaZG05,DingF07}.
Now consider a $k\times k'$ submatrix $\bfPhi$ of $\bfH$ corresponding to two subsets $\calD$ and $\calD'$ of $\calG$ which indicate rows and columns, respectively.
If both $\calD$ and $\calD'$ are difference sets of $\calG$, then $\bfPhi$ is a matrix with unimodular entries and equiangular rows and columns.
It is possible that $\bfPhi$ has rank $d$ where $d$ satisfies~\eqref{equation.extended ETF condition}.
For example, when $\calG=\bbZ_2^4$, numerical experimentation reveals we can take $\calD$ to be a McFarland difference set, and take $\calD'$ to be the complement of a distinct McFarland difference set, such as:
\begin{align*}
\calD&=\set{0000,0010,1000,1001,1100,1111},\\
\calD'&=\{0000,0001,0010,0100,1000,\\
&\hspace{15pt}\ 1001,1011,1100,1110,1111\}.
\end{align*}
The resulting $6\times 10$ matrix has unimodular entries and its columns and rows form an ETF for their $5$-dimensional spans in $\bbR^6$ and $\bbR^{10}$ respectfully.
Remarkably, numerical experimentation reveals other such ``paired" difference sets exist when $\calG=\bbZ_4^2$, but not when $\calG=\bbZ_2\times\bbZ_8$ or $\calG=\bbZ_2\times\bbZ_2\times\bbZ_4$,
despite the fact that difference sets of cardinality $6$ and $10$ exist in them all.
That is, the existence of these sets seems very sensitive to the group structure of $\calG$ itself.
We summarize this train of thought with the following open problem:
\begin{problem}
Given a finite abelian group $\calG$,
find all pairs of difference sets $\calD$, $\calD'$ of $\calG$ so that the corresponding submatrix $\bfPhi$ of the character table of $\calG$ satisfies the hypotheses of Theorem~\ref{theorem.extended ETF},
namely when $d=\operatorname{rank}(\bfPhi)$, $m=\abs{\calD}$ and $n=\abs{\calD'}$ satisfy \eqref{equation.extended ETF condition}.
\end{problem}

\section*{Acknowledgments}

This work was partially supported by NSF DMS 1321779, AFOSR F4FGA05076J002 and an AFOSR Young Investigator Research Program award.
The views expressed in this article are those of the authors and do not reflect the official policy or position of the United States Air Force, Department of Defense, or the U.S.~Government.

\begin{IEEEbiographynophoto}{Matthew Fickus} (M'08) received a Ph.D.\ in Mathematics from the University of Maryland in 2001.
In 2004, he joined the Department of Mathematics and Statistics at the Air Force Institute of Technology, where he is currently a Professor of Mathematics.
His research focuses on applying harmonic analysis and combinatorial design to problems of signal and image processing.
\end{IEEEbiographynophoto}

\begin{IEEEbiographynophoto}{Dustin G. Mixon} received a Ph.D.\ in Applied and Computational Mathematics from Princeton University in 2012.
He is currently an Assistant Professor of Mathematics in the Department of Mathematics and Statistics at the Air Force Institute of Technology.
His research interests include frame theory, compressed sensing, signal and image processing, and machine learning.
\end{IEEEbiographynophoto}

\begin{IEEEbiographynophoto}{John Jasper} received a Ph.D.\ in Mathematics from the University of Oregon in 2011.
He is currently a Visiting Assistant Professor in the Department of Mathematical Sciences at the University of Cincinnati.
His research focuses on operator theory and combinatorial design.
\end{IEEEbiographynophoto}

\begin{thebibliography}{WW}

\bibitem{AzarijaM15}
J.~Azarija, T.~Marc,
There is no (75,32,10,16) strongly regular graph,
preprint, arXiv:1509.05933.

\bibitem{BajwaCM12}
W.~U.~Bajwa, R.~Calderbank, D.~G.~Mixon,
Two are better than one: fundamental parameters of frame coherence,
Appl.\ Comput.\ Harmon.\ Anal.\ 33 (2012) 58-–78.

\bibitem{BandeiraFMW13}
A.~S.~Bandeira, M.~Fickus, D.~G.~Mixon, P.~Wong,
The road to deterministic matrices with the Restricted Isometry Property,
J.\ Fourier Anal.\ Appl.\ 19 (2013) 1123--1149.

\bibitem{BodmannH15}
B.~G.~Bodmann, J.~Hass,
Achieving the orthoplex bound and constructing weighted complex projective 2-designs with Singer sets,
submitted, arXiv:1509.05333.

\bibitem{Brouwer07}
A.~E.~Brouwer,
Strongly regular graphs,
in: C.~J.~Colbourn, J.~H.~Dinitz (Eds.), CRC Handbook of Combinatorial Designs (2007) 852-–868.

\bibitem{Brouwer15}
A.~E.~Brouwer,
Parameters of Strongly Regular Graphs,
http://www.win.tue.nl/$\sim$aeb/graphs/srg/

\bibitem{CorneilM91}
D.~Corneil, R. Mathon, eds.,
Geometry and combinatorics: Selected works of J.~J.~Seidel,
Academic Press, 1991.

\bibitem{Denniston69}
R.~H.~F.~Denniston,
Some maximal arcs in finite projective planes,
J.\ Combin.\ Theory 6 (1969) 317--319.

\bibitem{DingF07}
C.~Ding, T.~Feng,
A generic construction of complex codebooks meeting the Welch bound,
IEEE Trans.~Inform.~Theory 53 (2007) 4245--4250.

\bibitem{FickusJMP16}
M.~Fickus, J.~Jasper, D.~G.~Mixon, J.~Peterson,
Tremain equiangular tight frames,
arXiv:1602.03490 (2016).

\bibitem{FickusM15}
M.~Fickus, D.~G.~Mixon,
Tables of the existence of equiangular tight frames,
arXiv:1504.00253 (2015).

\bibitem{FickusMT12}
M.~Fickus, D.~G.~Mixon, J.~C.~Tremain,
Steiner equiangular tight frames,
Linear Algebra Appl.\ 436 (2012) 1014--1027.

\bibitem{GoethalsS70}
J.~M.~Goethals, J.~J.~Seidel,
Strongly regular graphs derived from combinatorial designs,
Can.\ J.\ Math.\ 22 (1970) 597--614.

\bibitem{HansenM92}
T.~Hansen, G.~L.~Mullen,
Primitive polynomials over finite fields,
Math.\ Comp.\ 59 (1992) 639--643.

\bibitem{HolmesP04}
R.~B.~Holmes, V.~I.~Paulsen,
Optimal frames for erasures,
Linear Algebra Appl.~377 (2004) 31--51.

\bibitem{JasperMF14}
J.~Jasper, D.~G.~Mixon, M.~Fickus,
Kirkman equiangular tight frames and codes,
IEEE Trans.\ Inform.\ Theory.\ 60 (2014) 170--181.

\bibitem{LemmensS73}
P.~W.~H.~Lemmens, J.~J.~Seidel,
Equiangular lines,
J.\ Algebra 24 (1973) 494--512.

\bibitem{LiO04}
S.~Li, H.~Ogawa,
Pseudoframes for subspaces with applications,
J.\ Fourier Anal.\ Appl.\ 10 (2004) 409--431.

\bibitem{Nguyen96}
N.~K.~Nguyen,
An algorithmic approach to constructing supersaturated designs,
Technometrics 38 (1996) 69--73.

\bibitem{Rankin56}
R.~A.~Rankin,
On the minimal points of positive definite quadratic forms,
Mathematika 3 (1956) 15--24.

\bibitem{Renes07}
J.~M.~Renes,
Equiangular tight frames from Paley tournaments,
Linear Algebra Appl.~426 (2007) 497--501.

\bibitem{RenesBSC04}
J.~M.~Renes, R. Blume-Kohout, A.~J.~Scott, C.~M.~Caves,
Symmetric informationally complete quantum measurements,
J.\ Math.\ Phys.\ 45 (2004) 2171--2180.

\bibitem{RyanB07}
K.~J.~Ryan, D.~A.~Bulutoglu,
$\rmE(s^2)$-optimal supersaturated designs with good minimax properties,
J.\ Statist.\ Plann.\ Inference 137 (2007) 2250--2262.

\bibitem{Strohmer08}
T.~Strohmer,
A note on equiangular tight frames,
Linear Algebra Appl.~429 (2008) 326–-330.

\bibitem{StrohmerH03}
T.~Strohmer, R.~W.~Heath,
Grassmannian frames with applications to coding and communication,
Appl.~Comput.~Harmon.~Anal.~14 (2003) 257--275.

\bibitem{SustikTDH07}
M.~A.~Sustik, J.~A.~Tropp, I.~S.~Dhillon, R.~W.~Heath,
On the existence of equiangular tight frames,
Linear Algebra Appl.~426 (2007) 619--635.

\bibitem{TangW97}
B.~Tang, C.~F.~J.~Wu,
A method for constructing supersaturated designs and its $Es^2$ optimality,
Canad.\ J.\ Statist.\ 25 (1997) 191--201.

\bibitem{Turyn65}
R.~J.~Turyn,
Character sums and difference sets,
Pacific J.\ Math.\ 15 (1965), 319--346.

\bibitem{Waldron09}
S.~Waldron,
On the construction of equiangular frames from graphs,
Linear Algebra Appl.\ 431 (2009) 2228--2242.

\bibitem{Wardlaw94}
W.~P.~Wardlaw,
Matrix representation of finite fields,
Math.\ Mag.\ 67 (1994) 289--293.

\bibitem{Welch74}
L.~R.~Welch,
Lower bounds on the maximum cross correlation of signals,
IEEE Trans.~Inform.~Theory 20 (1974) 397-–399.

\bibitem{XiaZG05}
P.~Xia, S.~Zhou, G.~B.~Giannakis,
Achieving the Welch bound with difference sets,
IEEE Trans.\ Inform.\ Theory 51 (2005) 1900--1907.

\bibitem{Yu15}
W.-H.~Yu,
There are no 76 equiangular lines in $\bbR^{19}$,
preprint, arXiv:1511.08569.

\bibitem{Zauner99}
G.~Zauner,
Quantum designs: Foundations of a noncommutative design theory,
PhD thesis, University of Vienna, 1999.
\end{thebibliography}
\end{document}